\documentclass[a4paper, 12pt]{article}

\usepackage[utf8]{inputenc}
\usepackage[T1]{fontenc}
\usepackage{fullpage}
\usepackage{comment}
\usepackage{amsmath, amsthm, amssymb}
\usepackage{graphicx}
\usepackage{enumerate}
\usepackage{authblk}
\usepackage{thmtools}
\usepackage{thm-restate}
\usepackage[colorlinks=true, citecolor=red]{hyperref}
\usepackage{float}

\usepackage{tikz}
\renewenvironment{abstract}
{\small\vspace{-1em}
\begin{center}
\bfseries\abstractname\vspace{-.5em}\vspace{0pt}
\end{center}
\list{}{
\setlength{\leftmargin}{0.6in}%
\setlength{\rightmargin}{\leftmargin}}%
\item\relax}
{\endlist}
\usepackage{comment}
\declaretheorem[name=Theorem, numberwithin=section]{theorem}
\declaretheorem[name=Lemma, sibling=theorem]{lemma}

\declaretheorem[name=Corollary, sibling=theorem]{corollary}
\declaretheorem[name=Conjecture, sibling=theorem]{conjecture}
\declaretheorem[name=Problem, sibling=theorem]{problem}
\declaretheorem[name=Claim, sibling=theorem]{claim}

\def\cqedsymbol{\ifmmode$\lrcorner$\else{\unskip\nobreak\hfil
\penalty50\hskip1em\null\nobreak\hfil$\lrcorner$
\parfillskip=0pt\finalhyphendemerits=0\endgraf}\fi}

\newcommand{\cG}{\mathcal{G}}

\usetikzlibrary{calc}
\let\le\leqslant
\let\ge\geqslant
\let\leq\leqslant
\let\geq\geqslant

\title{Recolouring planar graphs of \\girth at least five}

\author[1]{Valentin Bartier}
\author[2]{Nicolas Bousquet\thanks{Supported by ANR project GrR (ANR-18-CE40-0032)}}
\author[1]{Carl Feghali\thanks{Supported by the Czech Science Foundation under research grant 19-21082S and by
the French National Research Agency under research grant ANR DIGRAPHS ANR-19-CE48-0013-01.}}
\author[3]{Marc Heinrich}
\author[4]{Benjamin Moore\thanks{Supported by a NSERC doctoral grant.}}
\author[2]{Théo Pierron}

\affil[1]{Univ. Lyon, ENS Lyon, CNRS, LIP, F-69342, Lyon Cedex 07, France}
\affil[2]{Univ. Lyon, Université Lyon 1, LIRIS UMR CNRS 5205, F-69621, Lyon, France}
\affil[3]{School of Computing, University of Leeds, UK}
\affil[4]{Charles University, Institute of Computer Science, Prague, Czech Republic}

\date{\today}

\begin{document}

\maketitle

\begin{abstract}
For a positive integer $k$, the $k$-recolouring graph of a graph $G$ has as vertex set all proper $k$-colourings of $G$ with two $k$-colourings being adjacent if they differ by the colour of exactly one vertex. A result of Dyer et al. regarding graphs of bounded degeneracy implies that the $7$-recolouring graphs of planar graphs, the $5$-recolouring graphs of triangle-free planar graphs and the $4$-recolouring graphs planar graphs of girth at least six are connected. On the other hand, there are planar graphs whose $6$-recolouring graph is disconnected, triangle-free planar graphs whose $4$-recolouring graph is disconnected and planar graphs of any given girth whose $3$-recolouring graph is disconnected.  

The main result of this paper consists in showing, via a novel application of the discharging method, that the $4$-recolouring graph of every planar graph of girth five is connected. This completes the classification of the connectedness of the recolouring graph for planar graphs of given girth. We also prove some theorems regarding the diameter of the recolouring graph of planar graphs. 
\end{abstract}

\section{Introduction and results}

Let $k \geq 1$ be an integer, and let $G = (V, E)$ be a graph. A \emph{(proper) $k$-colouring} of $G$ is a function $\sigma : V \rightarrow \{ 1,\ldots,k \}$ such that, for every edge $xy\in E$, we have $\sigma(x)\neq \sigma(y)$. 
%Since we only consider proper colourings, we will then omit the proper for brevity. 
%The \emph{chromatic number} $\chi(G)$ of a graph $G$ is the smallest $k$ such that $G$ admits a $k$-colouring. 
The \emph{$k$-recolouring graph of $G$}, denoted by $\cG(G,k)$, is the graph whose vertices are the $k$-colourings of $G$, with two vertices being adjacent whenever the corresponding colourings differ on exactly one vertex.

The problem of determining whether $\cG(G,k)$ is connected is equivalent to that of determining whether the Glauber dynamics on the $k$-colourings of $G$ is ergodic (see e.g.~\cite{ChenDMPP19,frieze2007survey}). One of the earliest results in this direction is the following theorem by Dyer et al.~\cite{dyer2006randomly} and rediscovered by Cereceda et al.~\cite{Cereceda09}. For a positive integer $d$, a graph is $d$-degenerate if each of its subgraphs contains a vertex of degree at most $d$.

\begin{theorem}\label{thm:deg}
Let $k$ and $d$ be positive integers such that $k \geq d+2$. If $G$ is a $d$-degenerate graph, then $\cG(G,k)$ is connected. 
\end{theorem}

The bound on $k$ in Theorem \ref{thm:deg} is best possible (for example, if $G$ is the complete graph on $d+1$ vertices, then $\cG(G,d+1)$ a collection of $(d-1)!$ isolated vertices). Perhaps surprisingly, this bound also turns out to be best possible for planar graphs, triangle-free planar graphs and planar graphs of girth at least six. Indeed, 
by Euler's formula, planar graphs are $5$-degenerate, triangle-free planar graphs are $3$-degenerate and planar graphs of girth at least six are $2$-degenerate. Thus, Theorem~\ref{thm:deg} implies that if $G$ is a planar graph, then $\cG(G,7)$ is connected. Similarly, $\cG(G,5)$ is connected if $G$ is triangle-free and planar, and $\cG(G,4)$ is connected if $G$ has girth at least six and planar. 

On the other hand, there is a planar graph $G$ and a $6$-colouring of $G$ where the closed neighbourhood of each vertex contains all $6$ colours \cite{BonamyB18} and hence this $6$-colouring forms an isolated vertex in $\cG(G,6)$, implying $\cG(G,6)$ is disconnected. For triangle-free planar graphs, consider the $4$-colouring of the cube $Q_3$ where all colours are used in every face. Again, this colouring forms an isolated vertex in $\cG(Q_3,4)$ and hence $\cG(Q_3,4)$ is disconnected. 
 For planar graphs of girth at least six, the $3$-colouring of a cycle $C$ whose length is a multiple of three where the colours $1, 2, 3$ alternate in a cyclic fashion also forms an isolated vertex in $\cG(C,3)$. This motivates the following question. 

\begin{problem}\label{prob:main}
What is the smallest integer $\kappa$ such that for every planar graph $G$ of girth five, the graph $\cG(G,\kappa)$ is connected?
\end{problem}

Note by our earlier discussion that $\kappa \in \{4, 5\}$ for Problem \ref{prob:main}. In the main contribution of this paper, we settle the problem by showing that $\kappa = 4$. 

\begin{theorem}\label{thm:4col-girth5}
For every planar graph $G$ of girth five, the graph $\cG(G,4)$ is connected.  
\end{theorem}

By our earlier discussion, Theorem \ref{thm:4col-girth5} completes the classification of the connectedness of the recolouring graph for planar graphs of given girth. 
The proof idea of Theorem \ref{thm:4col-girth5} is based on a new application of the discharging method, in which we allow an infinite number of reducible configurations rather than the usual finite number. More precisely, we first prove that a special type of $C_5$ is forbidden. We then show that any possible gluing of these $C_5$ in a tree-like structure is also forbidden. We then conclude with a very short and simple discharging argument;  see Section~\ref{sub:defs} for a more detailed outline. 
\medskip

On a slightly different track, the question of determining what is the diameter of $\cG(G,k)$ (when it is connected) also received considerable attention. In other words, how fast can we reach one $k$-colouring of $G$ from another by changing the colour of one vertex at a time? This question and, particularly, the following strengthening of Theorem \ref{thm:deg} conjectured by Cereceda \cite{Cereceda} have been the subject of much attention in recent years. 

\begin{conjecture}
\label{Cerecedaconjecture}
Let $k, d$ be positive integers, $k \geq d + 2$ and let $G$ be a $d$-degenerate graph on $n$ vertices. Then $\cG(G,k)$ has diameter $O(n^2)$. 
\end{conjecture}

This bound would be best possible \cite{BonamyJ12}. Although the conjecture has resisted several efforts, there have been some partial results surrounding it. The most important breakthrough comes from Bousquet and Heinrich \cite{BousquetH19+}, who show amongst other results, that $\cG(G,k)$ has diameter $O(n^{d+1})$. Nonetheless, Conjecture \ref{Cerecedaconjecture} remains open even for $d = 2$. 

When $k$ is substantially larger than $d$, Bousquet and Perarnau \cite{BousquetP16} gave the following bound.

\begin{theorem}\label{thm:bousquet}
Let $k$ and $d$ be positive integers, such that $k \geq 2d + 2$. If $G$ is a $d$-degenerate graph on $n$ vertices, then $\cG(G,k)$ has diameter $O(n)$. 
\end{theorem}

We conjecture that the bound on $k$ in Theorem~\ref{thm:bousquet} can be lowered to $d+3$ and this would be best possible~\cite{BonamyJ12}.

\begin{conjecture}\label{conj:d+3}
Let $k, d$ be positive integers, $k \geq d + 3$ and let $G$ be a $d$-degenerate graph on $n$ vertices. Then $\cG(G,k)$ has diameter $O(n)$. 
\end{conjecture}

To the best of our knowledge, Conjecture~\ref{conj:d+3} is only known to hold for outerplanar graphs \cite{BartierBH21} and $1$-degenerate graphs. For partial results, Bartier and Bousquet~\cite{BousquetB19} proved that $\cG(G,d+4)$ has diameter $O(n)$ for every $d$-degenerate chordal graph $G$ of bounded maximum degree, and Dvo\v{r}\'ak and Feghali  proved that $\cG(G,10)$ has diameter $O(n)$ for every planar graph $G$ \cite{dvovrak2020update, dvovrak2021thomassen}, and if $G$ is triangle-free, then so does $\cG(G,7)$ \cite{dvovrak2021thomassen}. 

As a second contribution in this paper, we confirm the conjecture for planar graphs of girth at least $6$, improving the bound on $k$ in Theorem~\ref{thm:bousquet} in this special case.

\begin{theorem}\label{thm:5col-girth6}
For every planar graph $G$ on $n$ vertices of girth at least $6$, $\cG(G,5)$ has diameter $O(n)$.  
\end{theorem}

To prove Theorem \ref{thm:5col-girth6}, we show that planar graphs of girth $\geq 6$ contain a $2$-degenerate $1$-island of constant size (see Section~\ref{sec:5col-girth6} for a definition), by using virtually the same arguments as in \cite{EsperetO16} for a similar result; we then simply show that such a configuration is reducible for the problem. 
We remark that it may be possible to improve Theorem~\ref{thm:5col-girth6} by showing that $\cG(G,4)$ has linear diameter (this would be best possible, since, as discussed, there exists planar graphs $G$ of girth $\geq 6$ such that $\cG(G,3)$ is disconnected).

As our final contribution, we address a recent conjecture by Dvo\v{r}\'ak and Feghali that generalizes their aforementioned result that $\cG(G,10)$ has diameter $O(n)$ for every planar graph $G$. In order to state the conjecture and our theorem, we require some further definitions. A \emph{list assignment} $L$ for $G$ is a function that, to
each vertex $v \in V$, assigns a set $L(v)$ of colours. An \emph{$L$-colouring of $G$} is (proper) colouring $\varphi$ of $G$ such that $\varphi(v) \in L(v)$ for each $v \in V$. For a list assignment $L$ of $G$, the \emph{$L$-recolouring graph of $G$}, denoted by $\cG(G,L)$, is the graph whose vertices are the $L$-colourings of $G$, with two vertices being adjacent whenever the corresponding colourings differ on exactly one vertex. Dvo\v{r}\'ak and Feghali \cite{dvovrak2021thomassen} made the following conjecture.

\begin{conjecture}\label{conj:listcol}
Let $G = (V, E)$ be a planar graph on $n$ vertices, and let $L$ be a list assignment for $G$ with $|L(v)| \geq 10$ for each $v \in V$.  Then $\cG(G,L)$ has diameter $O(n)$. 
\end{conjecture}

In the case where the number $10$ of colours is replaced by $12$, Conjecture~\ref{conj:listcol} follows from [\cite{BousquetH19+}, Theorem 1]. We improve this result as follows: 

\begin{theorem}
\label{listrecolouring11}
Let $G = (V, E)$ be a planar graph on $n$ vertices,  and let $L$ a list assignment for $G$ where $|L(v)| \geq 11$ for every $v \in V$. Then $\cG(G,L)$ has diameter $O(n)$.
\end{theorem}

Our proof of Theorem \ref{listrecolouring11} follows the same spirit as our proof of Theorem \ref{thm:5col-girth6}, by combining a structural result of Borodin for planar graphs with a standard reducible configuration argument. 

\begin{table}[h]
    \centering
\begin{tabular}{|c|l|l|l|l|}
\hline Girth/colours & $4$ & $5$ & $6$ & $7$  \\
\hline $3$ & $+\infty$ & $+\infty$ & $+\infty$ & $O(n^6)$~\cite{BousquetH19+}  \\
\hline $4$ & $+\infty$ & $O(n^4)$~\cite{BousquetH19+} & $O(n \log^3(n))$~\cite{Feghali19+} & $O(n)$ \cite{dvovrak2021thomassen}  \\
\hline $5$ & $<+\infty$ (Thm~\ref{thm:4col-girth5})& $O(n \log^2 n)$~\cite{Feghali19+} & -  & -   \\
\hline $6$ & $O(n^3)$~\cite{BousquetH19+} & $O(n)$ (Thm~\ref{thm:5col-girth6}) & - & -  \\
\hline $7+$ & $O(n \log n)$~\cite{Feghali19+} & - & - & - \\
\hline
\end{tabular}    
\label{table}
\caption{Existing and open cases for the diameter of the $k$-recolouring graph of planar graphs with girth $g$ for some combinations of values of $k$ and $g$. Note that any bound at a given position in the table implies the same bound at its right and below it.}
\end{table}

We end this section with a summary in Table \ref{table} of some of the results on the diameter of $\cG(G, k)$  for planar graphs $G$ of given girth $g$. We should remark that there is no known lower bound beating the (trivial) $\Omega(n)$ bound on the diameter of such $\cG(G, k)$ whenever it is connected, which suggests an interesting direction of research. 

\medskip
For more results on this topic or on reconfiguration versions of decision problems other than graph colouring, we refer the reader to the surveys by van den Heuvel~\cite{Heuvel13} and by Nishimura~\cite{Nishimura17}.

\section{The proofs of Theorems \ref{thm:5col-girth6} and \ref{listrecolouring11}}
\label{sec:5col-girth6}

We start with some definitions. For a subgraph $H$ of a graph $G$ and a colouring $\sigma$ of $G$, we denote by $\sigma_{\restriction H}$ the restriction of $\sigma$ to $H$. A colouring $\sigma'$ of $G$ is obtained from $\sigma$ by \emph{a single step recolouring}, denoted $\sigma \sim \sigma'$, if $\sigma$ and $\sigma'$ differ in the colour of exactly one vertex. Let $\alpha$ and $\beta$ be two colourings of $G$.  Given a sequence $\sigma$ of recolourings $ \alpha_{\restriction H} \sim \dots \sim \beta_{\restriction H}$ in $H$, we say that $\sigma'$ \emph{lifts} to a sequence $\sigma$ of recolourings  $\alpha  \sim \dots \sim \beta$ in $G$ if $\sigma _{\restriction H = \sigma'}$. For a recolouring sequence $\alpha = \alpha_1 \sim \alpha_2 \sim \dots \sim \alpha_m = \beta$, the  set of \emph{new colours} of $H$ is the set $\bigcup_{i = 2}^m\{\alpha_i(u): \alpha_i(u) \not= \alpha_1(u), u \in H\}.$

In order to prove Theorems \ref{thm:5col-girth6} and \ref{listrecolouring11}, we will use the following basic, yet powerful, lemma (whose proof is implicit in a number of papers). 

%\textcolour{red}{As a note, this lemma works for list colouring as well, with no changes aside from $k$ becoming the size of the lists}

\begin{lemma}\label{lem:bestchoice}
Let $G$ be a graph and $v$ a vertex of $G$. Let $k$ and $C$ be positive integers, $k \ge d(v)+2$.  Let $\alpha$ and $\beta$ be $k$-colourings of $G$. Suppose that there is a recolouring sequence from $\alpha_{\restriction (G - v)}$ to $\beta_{\restriction (G - v)}$ that recolours the neighbourhood $N(v)$ of $v$ at most $C$ times. 
Then the sequence lifts to a sequence in $G$ from $\alpha$ to $\beta$ by at most $ \lceil \frac{C}{k-d(v)-1} \rceil+1$ recolourings of $v$. 
\end{lemma}

\begin{proof}
%There are $t = k - d_G(v) - 1$ colours not appearing on $v$ or any of its neighbours.

Since the total number of times the neighbours of $v$ change their colour is at most $C$, we can let $a_1, \dots, a_C$ be the new colours in $N(v)$ in order. Let $t = k - d_G(v) - 1$. We start by recolouring $v$ to a colour not in $\{a_1, \dots, a_t\}$, followed, in turn, by the first $t$ recolourings $a_1, \dots, a_t$ in $N(v)$. We then recolour $v$ to a colour not in $\{a_{t + 1}, \dots, a_{2t}\}$, followed by the next $t$ recolourings in $N(v)$; we continue this process until all $C$ recolourings in $N(v)$ have occurred. Notice that it is always possible to recolour $v$ since there are at least $t$ colours not appearing on $v$ or any of its neighbours.  Finally, we recolour $v$ to $\beta(v)$. Clearly, $v$ is recoloured at most $\lceil \frac{C}{t} \rceil+1$ times.
\end{proof}

To prove Theorem \ref{listrecolouring11}, we combine Lemma \ref{lem:bestchoice} with the following result of Borodin~\cite{Borodin}.

%\textcolour{magenta}{In this section we prove that given a plane graph $G$ and for any $11$-list-assignment $L$ of $G$, the diameter of the reconfiguration graph with respect to $L$ is linear. We need a theorem of Borodin.} 

\begin{theorem}\label{borodinthm}
If $G$ is a planar graph with minimum degree $5$, then $G$ has a $3$-face $T$ with vertices $v_{1},v_{2},v_{3}$ such that $d(v_{1}) + d(v_{2}) + d(v_{3}) \leq 17$. %Further, this inequality is sharp.
\end{theorem}

We say that $(G, \alpha, \beta)$ is a counterexample to Theorem \ref{listrecolouring11} if
\begin{itemize}
    \item $G$ is a planar graph,
    \item $\alpha$ and $\beta$ are $L$-colourings of $G$, where $L$ is as in Theorem \ref{listrecolouring11}, and
    \item $\alpha$ cannot be transformed to $\beta$ by at most $100$ recolourings per vertex.
\end{itemize}
A counterexample $(G, \alpha, \beta)$ to Theorem \ref{thm:5col-girth6} is said to be \emph{minimal} if $|V(G)|$ is minimum amongst all counterexamples to Theorem \ref{thm:5col-girth6}.

\begin{proof}[Proof of Theorem \ref{listrecolouring11}]

Let $(G, \alpha, \beta)$ be a minimal counterexample to Theorem \ref{listrecolouring11}. 

We first claim that $G$ has minimum degree at least $5$. Otherwise, $G$ has a vertex $v$ with degree at most $4$. By minimality, there is a recolouring sequence from $\alpha_{\restriction ({G-v})}$ to $\beta_{\restriction ({G-v})}$ where each vertex gets recoloured at most $100$ times. By Lemma \ref{lem:bestchoice}, this sequence lifts to a sequence in $G$ by at most $68$ recolourings of $v$, a contradiction.

We claim that $G$ does not have two vertices $v_{1}$ and $v_{2}$ such that $v_{1}v_{2} \in E(G)$ and the degrees of $v_{1}$ and $v_{2}$ are precisely $5$. Otherwise, by minimality, there is a recolouring sequence from $\alpha_{\restriction(G-\{v_{1}, v_{2}\})}$ to  $\beta_{\restriction (G-\{v_{1}, v_{2}\})}$ where each vertex gets recoloured at most $100$ times. By Lemma \ref{lem:bestchoice}, this sequence lifts to a sequence in $G-v_{1}$ by recolouring $v_{2}$ at most $68$ times. By the same lemma,  the latter sequence lifts to a sequence in $G$ by at most $95$ recolouring of $v_{1}$, which is a contradiction. This proves the claim.  

A similar argument can be applied to show that $G$ has no triangle $v_{1},v_{2},v_{3}$ such that $v_{1}$ has degree $5$ and both $v_{2}$ and $v_{3}$ have degree $6$ (consider the graphs $G-\{v_{1}, v_{2}, v_{3}\}$, $G - \{v_1, v_2\}$, $G - v_1$ and $G$ in order). Indeed, $v_3$ is recoloured $68$ times, $v_2$ is recoloured $95$ times and finally $v_1$ is recoloured at most $96$ times.

To complete the proof, by Theorem \ref{borodinthm}, $G$ contains a triangle $T$ with vertices $v_{1},v_{2},v_{3}$ such that $d(v_{1}) + d(v_{2}) + d(v_{3}) \leq 17$. As $G$ does not contain an edge where both vertices have degree $5$, it implies that up to relabelling, $d(v_{1}) = 5$, $d(v_{2}) = 6$ and $d(v_{3}) = 6$, which contradicts the preceding argument. 
\end{proof}

%\section{Linear diameter for planar graphs of girth at least 6}

The proof of Theorem \ref{thm:5col-girth6} also follows from Lemma \ref{lem:bestchoice}, but will require some more work. %remark that since planar graphs $G$ with girth at least $6$ are $2$-degenerate, by Theorem 1 in~\cite{BousquetP16} any $6$-colouring of $G$ can be transformed to any other $6$-colouring by $O(|V(G)|)$ recolourings. We improve the number of colours in the following theorem. 

 %\begin{theorem}
%\label{thm:5colg6}
%Let $G$ be a planar graph with girth at least $6$. Then $\mathcal{R}_5(G)$ has diameter at most $6143 \cdot n$.
%Let $\alpha$ and $\beta$ be two $5$-colourings of $G$. Then $\alpha$ can be transformed to $\beta$ by $O(|V(G)|)$ recolourings.  
%\end{theorem}

Given a graph $G$ and an induced subgraph $H$ of $G$, we say that $H$ is a \textit{$2$-degenerate $1$-island} if the following hold:
\begin{itemize}
    \item each vertex $v \in V(H)$ has at most one neighbour in $G-H$, and
    \item there exists an ordering $v_{1},\ldots,v_{|V(H)|}$ of the vertices of $H$ such that for each $i \in \{2,\ldots,|V(H)|\}$, the vertex $v_{i}$ has at most two neighbours in the graph $G - H + \{v_{1},\ldots,v_{i-1}\}$.
\end{itemize}

We say that $(G, \alpha, \beta)$ is a counterexample to Theorem \ref{thm:5col-girth6} if $G$ is a planar graph of girth at least $6$ and $\alpha$ and $\beta$ are $5$-colourings of $G$ that cannot be transformed to one another by at most $6143$ recolourings per vertex. A counterexample $(G, \alpha, \beta)$ to Theorem~\ref{thm:5col-girth6} is \emph{minimal} if $|V(G)|$ is minimum amongst all counterexamples to Theorem~\ref{thm:5col-girth6}. 

It turns out that a minimal counterexample cannot contain a $2$-degenerate $1$-island of order at most $12$. 

\begin{lemma}\label{lem:island}
Let $(G, \alpha, \beta)$ be a minimal counterexample to Theorem \ref{listrecolouring11}. Then $G$ does not contain a $2$-degenerate $1$-island $H$ where $|V(H)| \leq 12$.
\end{lemma}

\begin{proof}
Suppose towards a contradiction that $G$ contains a $2$-degenerate $1$-island $H$ with $|V(H)| \leq 12$. Let $t = |V(H)|$, and let $v_{1},\ldots,v_{t}$ be an ordering of $V(H)$ such that for all $i \in \{2,\ldots,t\}$ the vertex $v_{i}$ has at most two neighbours in the graph $G- H + \{v_{1},\ldots,v_{i-1}\}$. By minimality, there is a recolouring sequence $\sigma'$ in $G-H$ from $\alpha_{\restriction ({G-H})}$ to $\beta_{\restriction ({G-H})}$ that recolours every vertex at most $6143$ times. We  lift $\sigma'$ to a recolouring sequence $\sigma$ from $\alpha$ to $\beta$ in $G$ as follows. Let $c_{1} = 2049$, and for $i \geq 2$, let $c_{i} := \lceil \frac{6143 + c_{i-1}}{2}\rceil + 1 $. Note that $c_{12} = 6143$ and $c_{i+1} \geq c_{i}$. 

To prove the lemma, it suffices to show that $\sigma'$  lifts to $\sigma$ by recolouring each $v_{i}$ at most $c_{i}$ times. We proceed by induction on $i$. For $i =1$, since $v_{1}$ has at most one neighbour in $G-H$, Lemma \ref{lem:bestchoice} implies that $v_{1}$ is recoloured at most $2049= c_{1}$ times. Suppose  $i \in \{2,\ldots,t\}$. Since $H$ is a $2$-degenerate $1$-island,  $v_{i}$ has at most one neighbour in $G-H$ and at most two neighbours in $G - H + \{v_{1},\ldots,v_{i-1}\}$, the neighbours of $v_{i}$ are recoloured at most $c_{i-1} + 6143$ times. By Lemma \ref{lem:bestchoice}, $\sigma'$ lifts to a recolouring sequence of $G - H + \{1, \dots, v_i\}$ by recolouring $v_{i}$ at most $c_{i}$ times, as desired. 
\end{proof}

To finish the proof of Theorem \ref{thm:5col-girth6}, it suffices to show, by Lemma \ref{lem:island}, that every planar graph of girth at least $6$ contains a $2$-degenerate $1$-island $H$ where $|V(H)| \leq 12$. This follows nearly immediately from the proof of Theorem 8 in~\cite{EsperetO16}. We include the proof for completeness since the result stated in~\cite{EsperetO16} is slightly weaker (but holds for more general surfaces). 

\begin{lemma}
Every planar graph of girth at least $6$ contains a $2$-degenerate $1$-island $H$ where $|V(H)| \leq 12$.
\end{lemma}

\begin{proof}
Let $G$ be a vertex minimal counterexample. Then $G$ has minimum degree $2$; if not, the vertex of degree $1$ is a $2$-degenerate $1$-island. A similar argument can also be applied to show that $G$ does not contain

\begin{itemize}
\item[(A1)]  a path with at most $12$ vertices each of degree at most $3$ where the endpoints have degree $2$;
\item [(A2)] a cycle $C$ with at most $12$ vertices where all vertices in $C$ have degree at most $3$ and at least one vertex of $C$ has degree $2$. 
\end{itemize}

We now proceed via the classical discharging method. For  every $v \in V(G)$, let $\text{ch}(v) = 2d(v) -6$, and for every  face  $f$ in $G$, let $\text{ch}(f) = d(f) -6$. By Euler's formula, $\sum_{x \in V \cup F} \text{ch}(x) = -12$. 

 For every face $f$ in $G$, choose an orientation of $f$ and call it \emph{positive}, and let the other orientation be called \emph{negative}. We now describe a procedure by which charges are redistributed. For any vertex $v$, for any face $f$ incident to $v$ and for any orientation of $f$, take a maximal facial walk of $f$ 
starting at $v$ and going around $f$ in the prescribed orientation of $f$ such that the inner vertices of the walk have degree precisely $3$. Let $u$ be the other end vertex of the walk (note that $u=v$ is a possibility). Our discharging rules are as follows:

%CF either we remove the labels (B1) and (B2) or we use them more carefully. 
\begin{itemize}
\item If the walk has at least $5$ inner vertices, $f$ sends a charge of $\frac{1}{2}$ to $v$. 
\item Otherwise, $u$ sends a charge of $\frac{1}{2}$ to $v$.
\end{itemize}

 By the maximality of the walk, if the second case occurs, then by (A1) and (A2) $u$ has degree at least $4$ and there are no paths where both endpoints have degree $2$ and all internal vertices have degree at least $3$. 
 
To reach a contradiction, we show that the final charges $\text{ch}^*(f)$ and $\text{ch}^*(v)$ of each face $f$ and vertex $v$ is non-negative.
\medskip

\noindent
\textbf{Case 1: $d(v) =2$}. \\
 In this case, $v$ appears four times in the union of all boundary walks of faces of $G$ (for each face, we consider a boundary walk in the positive orientation and a boundary walk in the negative orientation of the face). Therefore $v$ receives a charge of $4 \times \frac{1}{2} = 2$, and thus $\text{ch}^*(v) = -2 + 2 \geq 0$, as required. 
 \medskip
 
 \noindent
\textbf{Case 2: $d(v) = 3$}. \\ 
Vertices of degree $3$ start with $0$ charge and do not give any charge, hence they end up with a final charge of $0$.
\medskip

\noindent
\textbf{Case 3: $d(v) \geq 4$}. \\
Consider the facial walks through which $v$ gives a charge of $\frac{1}{2}$ to  some vertices of degree $2$. Note that there are at most two such walks for each of the $d(v)$ faces incident to $v$. Therefore we have $\text{ch}^*(v)\geqslant 2d(v)-6-2\times d(v)\times\frac12=d(v)-6$, which is non-negative when $d(v)\geqslant 6$.

To handle the remaining cases, we first claim that if a neighbour $u$ of $v$ is immediately before $v$ in more than one such walk, then $u$ has degree $2$. Assume by contradiction that $u$ has degree $\geq 3$ and is just before $v$ in two facial walks ending at $v$, then by the above procedure $u$ must have degree exactly $3$ and there are two paths starting at $u$ each containing at most three inner vertices (each of degree $3$) and finishing at a vertex of degree $2$ (if there were more than three inner vertices, then together with vertices $u$ and $v$, it would give at least $5$ inner nodes, a contradiction since we assumed that we cannot apply the first discharging rule). It follows that $G$ contains a path on at most $9$ vertices each having degree at most $3$ and with its endpoints having degree $2$, contradicting (A1). 

Let us now complete the proof of Case 3. 
%Assume that $d(v) = 6$. Then $v$ is not adjacent to at least $5$ vertices of degree $2$, since otherwise $v$ together with these vertices form a $2$-degenerate $1$-island of size at most $6$. Thus  $\text{ch}^*(v) \geq 1$. 
If $d(v) = 5$, then $v$ cannot be adjacent to at least $4$ vertices of degree $2$ since otherwise $v$ together with these vertices form a $2$-degenerate $1$-island of size at most $5$. Therefore, by the claim, $v$ gives a charge of at most $3+2\times \frac 12=4$, implying $\text{ch}^*(v)\geq 0$. %And if $v$ is adjacent to at most two vertices of degree $2$, then $v$ gives a charge of at most $2 + 3 \times \frac{1}{2}$, implying $\text{ch}^*(v) \geq 0$. So we can assume that $v$ is adjacent to at least three vertices of degree $2$. Then $v$ cannot give a charge through its neighbours of degree at least $3$, since otherwise $G$ has a $2$-degenerate $1$-island of size at most $9$. Therefore $\text{ch}^*(v) \geq \text{ch}(v) - 3 \geq 0$. 

Suppose  $d(v) =4$. Then a similar argument as above can be applied to show that 
\begin{itemize}
\item $v$ is not adjacent to more than $2$ vertices of degree $2$, 
\item  if $v$ is adjacent to two vertices of degree $2$, then $v$ does not give
any charge through its neighbours of degree at least $3$,
\item  if $v$ is adjacent to one vertex of degree $2$, then $v$ does not give any charge through more than one neighbour of degree at least $3$

\end{itemize} 

In each case, we obtain $\text{ch}^*(v)\geq 0$. %And if  $v$ has one neighbour of degree $2$, then it cannot give charge through more than one neighbour of degree more than $2$, so $\text{ch}^*(v) \geq 1/2$. 
Lastly, if $v$ has no neighbour of degree $2$, then %$v$ does not send charge through more than two of its neighbours, as otherwise $G$ contains a $2$-degenerate $1$-island of size at most $12$. Thus 
by the claim, $\text{ch}^*(v) \geq 2-4\times 1/2\geq 0$. This completes Case 3. 
\medskip

\noindent
\textbf{Case 4: $d(f) = 6$}.\footnote{The degree of a face being the number of vertices incident to it.} \\ In this case, no vertex receives any charge from $f$, since otherwise $f$ has a vertex of degree $2$ and five vertices of degree $3$, which implies that the face is a $2$-degenerate $1$-island of order $6$. 
\medskip

\noindent
\textbf{Case 5: $d(f) \geq 7$}. \\ For each vertex of degree $2$ that receives a charge of $\frac{1}{2}$ from $f$ in the positive orientation, let $A^{+}(v)$ be the set consisting of the five vertices of degree exactly $3$ following $v$ in the positive orientation of $f$ (notice that these vertices exist). We define $A^{-}(v)$ similarly for the negative orientation of $f$. Then $A^{+}(v)$ and $A^{-}(u)$ are pairwise disjoint, otherwise there is a $2$-degenerate $1$-island of size at most $11$. Therefore each face $f$ sends a charge of at most $2 \times \frac{1}{2}\times \lfloor \frac{d(f)}{6} \rfloor$, and thus as $d(f) \geq 7$, we get that $\text{ch}^*(f) \geq 0$. This completes Case 5 and hence the proof of the theorem. 
\end{proof}

\section{The proof of Theorem \ref{thm:4col-girth5}}
\label{sec:4col-girth5}
\subsection{Outline of the proof}
\label{sub:defs}
We say that $(G, \alpha, \beta)$ is a \emph{special counterexample} to Theorem \ref{thm:4col-girth5} if

\begin{itemize}
    \item $G$ is a planar graph of girth $5$,
    \item $\alpha$ is a $4$-colouring of $G$ and $\beta$ is a $3$-colouring of $G$, and
    \item there is no sequence of recolourings from $\alpha$ to $\beta$.
\end{itemize}

Note that every special counterexample is indeed a counterexample to Theorem~\ref{thm:4col-girth5}. Moreover, because of the celebrated Grötszch's Theorem \cite{grotzsch1959} (stating that every triangle-free planar is $3$-colourable), if there is no special counterexample, then Theorem~\ref{thm:4col-girth5} holds. Therefore, in order to prove Theorem~\ref{thm:4col-girth5}, it is sufficient to show that there is no special counterexample. For clarity, from now on, we will omit the adjective ``special''.

%As a side note, our argument can (with some rephrasing) in fact be used to show that every planar graph of girth at least $5$ is $3$-colourable and thus there is no need to appeal to Grötszch's Theorem. The idea being that we could change the theorem statement to prove that every $4$-colouring of a planar graph of girth at least $5$ reconfigures to a fixed $3$-colouring, with the same proof going through in the same way. We now give a sketch of the proof.  %With this statement, arguing by minimality (see below for a definition) when we delete a configuration, the smaller graph will have all $4$-colourings reconfigure to some fixed $3$-colouring. Proving this statement necessarily proves that planar graphs of girth at least five are $3$-colourable, and since the configurations needed for this are the same as what we give in our proof, our proof would additionally give this result.  

We say that a counterexample $(G, \alpha, \beta)$ is \emph{minimal} if $|V(G)| \leq |V(H)|$ for any other counterexample $(H, \alpha', \beta')$. We prove Theorem~\ref{thm:4col-girth5} by contradiction, showing that no minimal counterexample exists. %efore giving the details, we sketch the main idea used. 

\begin{comment} 
Clearly, Theorem \ref{thm:4col-girth5} will follow if we can show that there is no such counterexample. We include a proof for completeness. 

\begin{lemma}
If there is no counterexample to Theorem \ref{thm:4col-girth5}, then Theorem \ref{thm:4col-girth5} holds.  
\end{lemma}

\begin{proof}
Let $G$ be a planar graph of girth $5$, and let $\alpha$ and $\beta$ be $4$-colourings of $G$. To prove the lemma, it suffices to show that $\alpha$ can be transformed to $\beta$. 

By Grötszch's Theorem, $G$ admits a $3$-coloring $\varphi$. By the premise of the lemma, $\alpha$ and $\beta$ can each be transformed to $\varphi$. This implies the lemma.  
\end{proof}
Thus, in the remainder of the proof, we focus on showing that there is no counterexample to Theorem \ref{thm:4col-girth5}.
\end{comment}

For a face $f$ of $G$, the set of vertices incident with $f$ is denoted $V(f)$.
Let $v$ be a vertex of $G$ of degree exactly $4$ and let $f$ be a $5$-face incident with $v$. A $5$-face $f' \not=f$ of $G$ is \emph{opposite to $f$ with respect to $v$} if $V(f) \cap V(f') = \{v\}$ (note that since $v$ has degree $4$ and $G$ has girth $5$, $f'$ is well defined). We say that a vertex $u$ incident with a $5$-face $f$ is \emph{bad for $f$} if either \begin{itemize}
\item $u$ has degree $3$,  or 
\item $u$ has degree $4$ and $f$ has an opposite $5$-face $f'$ with respect to $u$ such that each vertex in $V(f') - \{u\}$ is bad for $f'$. 
\end{itemize}

If $v$ is bad for $f$, we say that $(v,f)$ is a \emph{bad pair}. Observe that bad pairs are well-defined inductively, and that the definition gives a natural quasi-order on bad pairs.

To prove Theorem~\ref{thm:4col-girth5}, we show that every minimal counterexample cannot contain two infinite families of forbidden structures, namely:
\begin{itemize}
    \item a 5-face $f$ where all vertices are bad for $f$ (Lemma~\ref{lem:config1})
    \item a 5-vertex $v$ adjacent to four 5-faces $f_1,f_2,f_3,f_4$ such that each vertex of $f_i$ (except $v$) is bad for $f_i$ (Lemma~\ref{lem:config2}).
\end{itemize}
We then apply a simple discharging argument to reach a contradiction with the existence of a minimal counterexample.

\subsection{Structure of a minimal counterexample}

In this section, we show a series of lemmas leading to the sought forbidden structures from Lemmas~\ref{lem:config1} and~\ref{lem:config2} (see the end of this section).  

\begin{lemma}\label{lem:deg2}
Let $(G, \alpha, \beta)$ be a minimal counterexample. Then $G$ is connected and has minimum degree at least $3$.
\end{lemma}

\begin{proof}
Clearly, $G$ is connected. Suppose that $G$ has a vertex $v$ of degree at most $2$. By minimality of $(G, \alpha, \beta)$, there is a recolouring sequence from $\alpha_{ \restriction (G - v)}$ to $\beta_{\restriction (G - v)}$. This sequence can be lifted to a sequence in $G$ by first recolouring $v$ whenever a neighbour of $v$ is recoloured to the colour of $v$ (this is possible since the number $4$ of colours implies there is always a colour not appearing on $v$ or any of its neighbours). At the end of the sequence, we recolour $v$ to $\beta(v)$, which is a contradiction. 
\end{proof}

We say that a vertex $v$ of $G$ is \emph{$\alpha$-frozen} if $\alpha(v) = \gamma(v)$ for \emph{every} colouring $\gamma$ obtainable from $\alpha$ by a sequence of recolourings. We say that $v$ is \emph{$\alpha$-locked} if all colours $\{1, 2, 3, 4\}$ appear in the closed neighbourhood of $v$. Notice that if a vertex is frozen, then it is also locked, but the converse is not necessarily true. A vertex is said to be \emph{$\alpha$-loose} if it is not $\alpha$-frozen and \emph{$\alpha$-free} if it is not $\alpha$-locked.  A locked vertex is said to be \emph{unlocked} if it becomes free after recolouring at least one of its neighbours. 

We have the following simple observations about frozen vertices. The first lemma is obvious. 

\begin{lemma}
\label{rmk:unfreeze-propagate}
Let $G$ be a graph, and let $\varphi$ be a $4$-colouring of $G$. If $v$ is a vertex of $G$ of degree $3$ that is $\varphi$-locked, then recolouring one of the neighbours of $v$ unlocks $v$. 
\end{lemma}

The following simple consequence of the lemma will be used repeatedly. 

\begin{lemma}\label{lem:extra}
Let $G$ be a graph, and let $\varphi$ be a $4$-colouring of $G$. Let $P = v_1v_2\dots v_k$ be a path in $G$ such that $v_1$ is $\varphi$-free and, for $i \in \{2, \dots, \textcolor{violet}{k}\}$, $v_i$ is $\varphi$-locked and has degree $3$. Then, for $j \in \{2, \dots, \textcolor{violet}{k}\}$ there is a sequence of recolourings from $\varphi$ to some colouring of $G$ that unlocks $v_j$ that only recolour vertices of $P$. 
\end{lemma}

\begin{proof}
We recolour $v_i$, $1 \leq i \leq j - 1$ in order, i.e. starting with $v_1$ and moving towards $v_{j - 1}$. This is possible by Lemma \ref{rmk:unfreeze-propagate}. 
\end{proof}

\begin{lemma}
\label{lem:frozen3}
Let $G$ be a graph, $\varphi$ be a $4$-colouring of $G$ and $v$ be a vertex of $G$ with degree $3$ that is $\varphi$-frozen. Then every neighbour of $v$ is $\varphi$-frozen. 
\end{lemma}
\begin{proof}
Suppose by contradiction that there is a neighbour $w$ of $v$ that is $\varphi$-loose. By definition of $w$, there is a recolouring sequence starting from $\varphi$ which recolours $w$. Up to truncating the recolouring sequence and renaming $w$, one can assume that $w$ is the only recoloured neighbour of $v$.  Since $v$ is $\varphi$-frozen and has degree precisely $3$, recolouring $w$ in turn unlocks $v$ by Lemma \ref{rmk:unfreeze-propagate}, which is a contradiction.  %Let us consider the first moment that a vertex in $N(v)$ is recoloured by this sequence. Before the recolouring, since $v$ is frozen, it is locked, and the three neighbours of $v$ have different colours. However, after this recolouring there can be only two different colours in the neighbourhood of $v$, which contradicts the assumption that $v$ is frozen.
\end{proof}

%We will also need a similar property for vertices of degree $4$:
\begin{lemma}
\label{lem:frozen4}
Let $G$ be a graph, $\varphi$ be a $4$-colouring of $G$ and $v$ be a vertex of $G$ with degree $4$ that is $\varphi$-frozen. If $v$ has two frozen neighbours with the same colour in $\varphi$, then all neighbours of $v$ are $\varphi$-frozen.
\end{lemma}
\begin{proof}
Let $u_1, u_2$ be the two $\varphi$-frozen neighbours of $v$ such that $\varphi(u_1) = \varphi(u_2)$, and let $w_1$ and $w_2$ be the other two neighbours of $v$. Suppose for a contradiction that one of $w_1, w_2$, say $w_1$, is $\varphi$-loose. So we can let $\sigma$ be the first colouring obtainable from $\varphi$ via a sequence of recolourings such that $\sigma(w_1) \not= \varphi(w_1)$. 

Note that $v$  being $\varphi$-frozen implies $\varphi(v), \varphi(u_1), \varphi(w_1), \varphi(w_2)$ are pairwise distinct. Thus, $\sigma(w_1) \in \{ \varphi(u_1), \varphi(w_1), \varphi(w_2)\}$ and since $u_1$ and $u_2$ are frozen, $\sigma(u_i) = \varphi(u_i)$ for $i=1, 2$. Hence, $v$ has only two colours appearing in its neighbourhood, and hence is $\sigma$-loose, which is a contradiction. 
\end{proof}

%\subsection{\textcolour{magenta}{Four degree three} vertices incident to 5-faces}

Using these properties, we can prove the following result. Let $G$ be a plane graph with girth $5$, and let $f$ be a face of $G$. The set of neighbours adjacent to some vertex in $V(f)$ (but not on $f$) is denoted by $N(f)$.  A $5$-face $f$ in $G$ with $V(f) = \{v_1, v_2, v_3,v_4, v_5\}$ is \emph{bad} if $d(v_1) \leq 4$ and $d(v_i) = 3$ for $i \in \{2, \dots, 5\}$; it is \emph{very bad} if moreover $d(v_1) = 4$. 

\begin{lemma}
\label{lem:5-face}
Let $(G, \alpha, \beta)$ be a minimal counterexample. Then $G$ does not contain a bad $5$-face with at least one vertex that is $\alpha$-loose. 
\end{lemma}
\begin{proof}
Suppose by contradiction that $G$ contains a bad $5$-face $f$ with an $\alpha$-loose vertex. By definition, $\alpha$ can be transformed to some $4$-colouring $\varphi$ such that $f$ has a vertex that is $\varphi$-free. 

For a colouring $\sigma$ of $G$, we shall slightly abuse notation by saying the face $f$ is $\sigma$-free if at least one vertex incident to $f$ is $\sigma$-free. Our aim is to show that any recolouring sequence of $G - V(f)$ can be lifted into a sequence in $G$, which is a a contradiction to the minimality of $(G, \alpha, \beta)$. 

 Let $u_1$ denote the vertex of degree at most $4$ in $f$, and let $u_2, \ldots, u_5$ denote the vertices of degree precisely $3$ on $f$ in a clockwise ordering starting at $u_1$ around $f$. By Lemma \ref{lem:deg2}, $d(u_1) \geq 3$. We require the following two claims.

\begin{claim}
\label{claim:unfreeze}
If $f$ is $\sigma$-free and $w \in \{u_1, \dots, u_5\}$, then there is a recolouring sequence starting from $\sigma$ that recolours only vertices in $V(f)$ and unlocks $w$.  
\end{claim}

\begin{proof}[Proof of Claim \ref{claim:unfreeze}.]
 If $w$ is $\sigma$-free, there is nothing to prove. 
 
 So we can assume that $w$ is $\sigma$-locked. If $w\neq u_1$, then one of the two facial paths from $w$ to $u_1$ goes through a free vertex (otherwise all vertices of $f$ are locked). Let $u$ be the first such vertex. 
 We can apply Lemma~\ref{lem:extra} with the path from $w$ to $u$, and the claim follows.
% We can then successively apply Lemma~\ref{rmk:unfreeze-propagate} to each vertex on the facial path from $u$ to $w$ that avoids $u_1$, and the claim follows. 
 
 %Suppose first that  $u_1$ is $\sigma$-free, thus $w \in \{u_2, \dots, u_5\}$. Let $u_j$ be the first $\sigma$-free vertex encountered if one traverses $f$ anticlockwise from $w = u_i$ ($ 2 \leq i \leq 5)$ towards $u_1$. Since $u_1$ is $\sigma$-free, such a vertex exists. By successively applying Lemma~\ref{rmk:unfreeze-propagate}, recolouring $u_k$ unlocks $u_{k+1}$ for $j \leq k < i$ in order, and the claim follows. 
    
Therefore we can assume that $w = u_1$. Since $f$ is $\sigma$-free and each $u_i \not= w$ has degree precisely $3$, the same argument can be applied to show that for $i \in \{2, \dots, 5\}$ each $u_i$ can be unlocked by recolouring vertices in $V(f) - w$ and if $d(u_1) = 3$, then $w = u_1$ may also be unlocked. To complete the proof, it remains to show that $w$ can also be unlocked if $d(w) = 4$ by recolouring vertices in $V(f)$. %Hence, our goal is to construct a recolouring sequence which ends with recolouring the vertex $v$. 

Let $w_1$ and $w_2$ denote the two neighbours of $w$ outside $f$, and suppose $\sigma(w) = 1$, $\sigma(w_1) = 3$ and $\sigma(w_2) = 4$. Our aim is to try to recolour $w$ to colour $2$ (by possibly first recolouring only vertices in $V(f) - w$).  If $\sigma(u_2), \sigma(u_5) \neq 2$, then $w$ is $\sigma$-free,  a contradiction. Therefore, either $\sigma(u_2) = 2 \not= \sigma(u_5)$ or $\sigma(u_2) = \sigma(u_5) = 2$. We address the two cases separately.   %
%and move to the next step. Otherwise, we distinguish two cases:

\medskip

\noindent
\textbf{Case 1:} $\sigma(u_2) = 2 \not= \sigma(u_5)$. \\ Since $f$ is $\sigma$-free,  we can let $i \in \{2, \dots, 5\}$ be the smallest index such that $u_i$ is $\sigma$-free. If $i = 2$, then we recolour $u_2$ which, in turn, unlocks $v$. And if $i \neq 4$, then recolouring $u_j$ unlocks $u_{j-1}$ for $i \geq j \geq 1$ in order as needed. 

It only remains to address the case $i = 5$. We first try to recolour $u_5$ with a colour distinct from $2$; this in turn unlocks $u_4$ and by recolouring $u_4$ vertex $u_3$ is unlocked etc. until $u_1$ is unlocked. So we can assume that we can only recolour $u_5$ with colour $2$.  Without loss of generality, set $\sigma(u_5) = 3$. Since $u_4$ is locked, $\sigma(u_3) \neq 3 = \sigma(u_5)$ and since $u_2$ is locked, $\sigma(u_3) \not= 1 = \sigma(u_1)$. Therefore, $\sigma(u_3) = 4 \not= \sigma(u_2) = 2$. Since $u_3$ is locked, $\sigma(u_4) \not=  2 = \sigma(u_2)$ and thus $\sigma(u_4) = 1$. By applying the recolouring sequence shown in Figure~\ref{fig:recol-vertex-v1}, $w$ may be recoloured to $2$. This completes Case~1. 
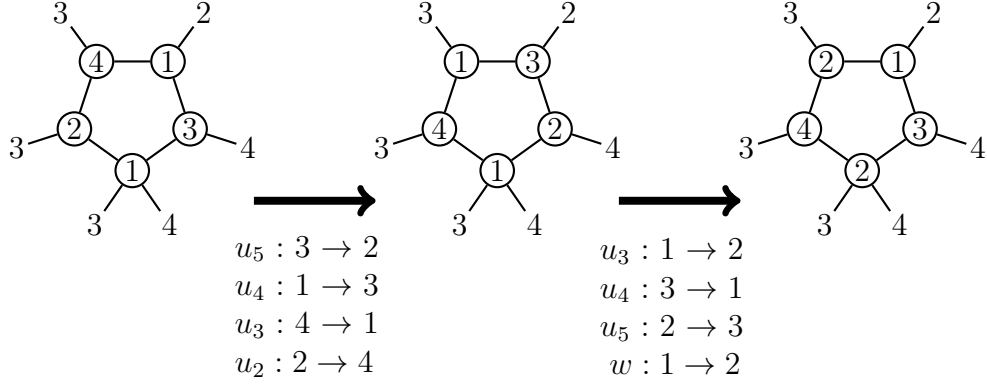
\begin{figure}
    \centering
    \begin{tikzpicture}[scale=.8,thick,
    node/.style = {inner sep = 1pt},
    cnode/.style = {node, draw, circle}
    ]
        \node[cnode] (v) at (270:1) {1} ; 
        \node[cnode] (u1) at (198:1) {2} ;
        \node[cnode] (u2) at (126:1) {4} ;
        \node[cnode] (u3) at (54:1) {1} ;
        \node[cnode] (u4) at (342:1) {3} ;
        \draw (v) -- (u1) -- (u2) -- (u3) -- (u4) -- (v) ;
        \node[node] (vb1) at (252:2) {3} ; 
        \node[node] (vb2) at (288:2) {4} ; 
        \node[node] (ub1) at (198:2) {3} ; 
        \node[node] (ub2) at (126:2) {3} ; 
        \node[node] (ub3) at (54:2) {2} ; 
        \node[node] (ub4) at (342:2) {4} ; 
        
        \draw (vb1) -- (v) -- (vb2) ;
        \foreach \i in {1,2,3,4}{
            \draw (u\i) -- (ub\i) ;
        }

        \node[text width = 2.5cm] at (3.25, -3.25) {
            $u_5 : 3 \rightarrow 2$ 
            $u_4 : 1 \rightarrow 3$
            $u_3 : 4 \rightarrow 1$
            $u_2 : 2 \rightarrow 4$
        } ; 
        
        \draw[->, line width = 3pt] (2, -1.5) -- (4, -1.5) ;

        \begin{scope}[shift={(6, 0)}]
        \node[cnode] (v) at (270:1) {1} ; 
        \node[cnode] (u1) at (198:1) {4} ;
        \node[cnode] (u2) at (126:1) {1} ;
        \node[cnode] (u3) at (54:1) {3} ;
        \node[cnode] (u4) at (342:1) {2} ;
        \draw (v) -- (u1) -- (u2) -- (u3) -- (u4) -- (v) ;
        \node[node] (vb1) at (252:2) {3} ; 
        \node[node] (vb2) at (288:2) {4} ; 
        \node[node] (ub1) at (198:2) {3} ; 
        \node[node] (ub2) at (126:2) {3} ; 
        \node[node] (ub3) at (54:2) {2} ; 
        \node[node] (ub4) at (342:2) {4} ; 
        
        \draw (vb1) -- (v) -- (vb2) ;
        \foreach \i in {1,2,3,4}{
            \draw (u\i) -- (ub\i) ;
        }
        \end{scope}
        \node[text width = 2.5cm] at (9.25, -3.25) {
            $u_3 : 1 \rightarrow 2$ 
            $u_4 : 3 \rightarrow 1$
            $u_5 : 2 \rightarrow 3$
            \hspace*{0.3em}$w : 1 \rightarrow 2$
        } ; 
        
        \draw[->, line width = 3pt] (8, -1.5) -- (10, -1.5) ;

        \begin{scope}[shift={(12, 0)}]
        \node[cnode] (v) at (270:1) {2} ; 
        \node[cnode] (u1) at (198:1) {4} ;
        \node[cnode] (u2) at (126:1) {2} ;
        \node[cnode] (u3) at (54:1) {1} ;
        \node[cnode] (u4) at (342:1) {3} ;
        \draw (v) -- (u1) -- (u2) -- (u3) -- (u4) -- (v) ;
        \node[node] (vb1) at (252:2) {3} ; 
        \node[node] (vb2) at (288:2) {4} ; 
        \node[node] (ub1) at (198:2) {3} ; 
        \node[node] (ub2) at (126:2) {3} ; 
        \node[node] (ub3) at (54:2) {2} ; 
        \node[node] (ub4) at (342:2) {4} ;

        \draw (vb1) -- (v) -- (vb2) ;
        \foreach \i in {1,2,3,4}{
            \draw (u\i) -- (ub\i) ;
        }
        \end{scope}

    \end{tikzpicture}
    \caption{\label{fig:recol-vertex-v1} Recolouring sequence from Case 1 of Claim~\ref{claim:unfreeze}.%Possible recolouring sequence to change the colour of $v$ from $1$ to $2$.
    }
    
\end{figure}

\medskip
    
\noindent
\textbf{Case 2:} $\sigma(u_2) = \sigma(u_5) = 2$, \\ 
If either $u_2$ or $u_5$, say $u_2$, is $\sigma$-free, then we recolour $u_2$ with a colour distinct from $2$ and apply Case 1.  

Otherwise, since $f$ is $\sigma$-free, either  $u_3$ or $u_4$ is $\sigma$-free. Assume without loss of generality that $u_3$ is $\sigma$-free; then recolouring $u_3$ unlocks $u_2$, in which case the argument from the preceding paragraph can be applied. This completes Case 2. 
\end{proof}

Claim~\ref{claim:unfreeze} will allow us to prove the following claim, from which we will derive Lemma~\ref{lem:5-face}. 

\begin{claim}
\label{claim:extend}
If  $f$ is $\sigma$-free, $z \in V(f)$ and $a \in \{1, 2, 3, 4\}$, then there is a recolouring sequence from $\sigma$ to a colouring $\sigma'$ that recolours only vertices in $V(f)$ such that $\sigma'(z) \not= a$ and some vertex in $\{u_2, \dots, u_5\}$ is $\sigma'$-free.
\end{claim}
\begin{proof}[Proof of Claim \ref{claim:extend}.]
We distinguish three cases (the cases $z=u_4$ and $z=u_5$ are symmetric to respectively $z=u_3$ and $z=u_2$).

\medskip

\noindent
\textbf{Case 1:} $z = u_1$. \\ 
By Claim~\ref{claim:unfreeze}, we can transform $\sigma$ to some colouring $\sigma'$ by only recolouring vertices in $V(f)$ so that $z$ is $\sigma$-free and so by recolouring $z$ if necessary we can further assume $\sigma'(z) \not=a$. If some vertex in $ \{u_1, \dots, u_5\} \setminus \{z\}$ is $\sigma'$-free, then the claim follows. So we can assume that each $u_i$ is $\sigma'$-locked. By symmetry, we assume that the neighbours of $z$ outside $f$ are coloured with 2 and 3. Thus, $\sigma'(u_2) = \sigma'(u_5) = 2$ and so $\sigma'(u_3) = 3, \sigma'(u_4) = 4 \not= \sigma'(z) = 1$. 
  
 %Set $\sigma'(w) = 1$. Since all the $u_i$ are frozen, then this implies that $\sigma_1(u_1) = \sigma_1(u_4)$, and up to symmetry, we can assume that they are both of colour $2$. Finally, the vertices $u_2$ and $u_3$ cannot be coloured $1$ since otherwise either $u_1$ or $u_4$ would not be locked. Up to symmetry, we can assume that $\sigma_1(u_2) = 3$ and $\sigma_1(u_3) = 4$. 
 
 Now we proceed with the recolouring sequence shown in Figure~\ref{fig:recol-vertex-v} (note that, at the end of the sequence, the colour of $z$ is $1$ while $u_2$ is free, as needed). 
    
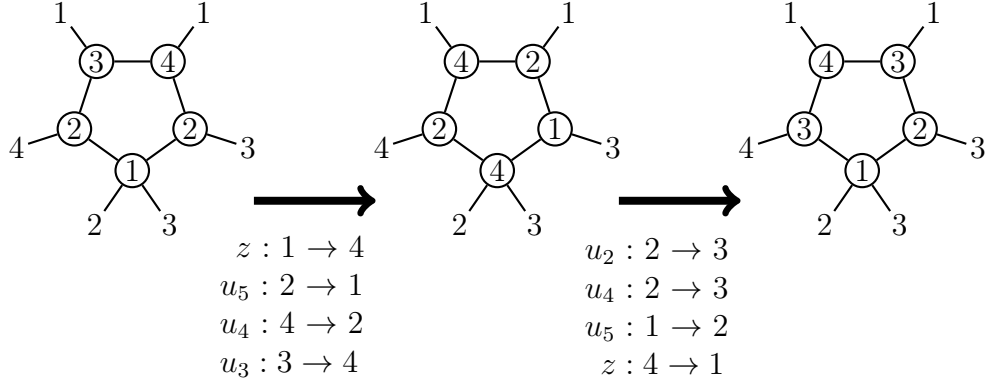
\begin{figure}
    \centering
    \begin{tikzpicture}[thick, scale=0.8,
    node/.style = {inner sep = 1pt},
    cnode/.style = {node, draw, circle}
    ]
        \node[cnode] (v) at (270:1) {1} ; 
        \node[cnode] (u1) at (198:1) {2} ;
        \node[cnode] (u2) at (126:1) {3} ;
        \node[cnode] (u3) at (54:1) {4} ;
        \node[cnode] (u4) at (342:1) {2} ;
        \draw (v) -- (u1) -- (u2) -- (u3) -- (u4) -- (v) ;
        \node[node] (vb1) at (252:2) {2} ; 
        \node[node] (vb2) at (288:2) {3} ; 
        \node[node] (ub1) at (198:2) {4} ; 
        \node[node] (ub2) at (126:2) {1} ; 
        \node[node] (ub3) at (54:2) {1} ; 
        \node[node] (ub4) at (342:2) {3} ; 
        
        \draw (vb1) -- (v) -- (vb2) ;
        \foreach \i in {1,2,3,4}{
            \draw (u\i) -- (ub\i) ;
        }

        \node[text width = 2.5cm] at (3, -3.25) {
            \hspace*{0.5em}$z : 1 \rightarrow 4$ 
            $u_5 : 2 \rightarrow 1$
            $u_4 : 4 \rightarrow 2$
            $u_3 : 3 \rightarrow 4$
        } ; 
        
        \draw[->, line width = 3pt] (2, -1.5) -- (4, -1.5) ;

        \begin{scope}[shift={(6, 0)}]
        \node[cnode] (v) at (270:1) {4} ; 
        \node[cnode] (u1) at (198:1) {2} ;
        \node[cnode] (u2) at (126:1) {4} ;
        \node[cnode] (u3) at (54:1) {2} ;
        \node[cnode] (u4) at (342:1) {1} ;
        \draw (v) -- (u1) -- (u2) -- (u3) -- (u4) -- (v) ;
        \node[node] (vb1) at (252:2) {2} ; 
        \node[node] (vb2) at (288:2) {3} ; 
        \node[node] (ub1) at (198:2) {4} ; 
        \node[node] (ub2) at (126:2) {1} ; 
        \node[node] (ub3) at (54:2) {1} ; 
        \node[node] (ub4) at (342:2) {3} ; 
        
      \draw (vb1) -- (v) -- (vb2) ;
        \foreach \i in {1,2,3,4}{
            \draw (u\i) -- (ub\i) ;
        }
        \end{scope}
        \node[text width = 2.5cm] at (9, -3.25) {
            $u_2 : 2 \rightarrow 3$ 
            $u_4 : 2 \rightarrow 3$
            $u_5 : 1 \rightarrow 2$
            \hspace*{0.5em}$z : 4 \rightarrow 1$
        } ; 
        
        \draw[->, line width = 3pt] (8, -1.5) -- (10, -1.5) ;

        \begin{scope}[shift={(12, 0)}]
        \node[cnode] (v) at (270:1) {1} ; 
        \node[cnode] (u1) at (198:1) {3} ;
        \node[cnode] (u2) at (126:1) {4} ;
        \node[cnode] (u3) at (54:1) {3} ;
        \node[cnode] (u4) at (342:1) {2} ;
        \draw (v) -- (u1) -- (u2) -- (u3) -- (u4) -- (v) ;
        \node[node] (vb1) at (252:2) {2} ; 
        \node[node] (vb2) at (288:2) {3} ; 
        \node[node] (ub1) at (198:2) {4} ; 
        \node[node] (ub2) at (126:2) {1} ; 
        \node[node] (ub3) at (54:2) {1} ; 
        \node[node] (ub4) at (342:2) {3} ; 
        
       \draw (vb1) -- (v) -- (vb2) ;
        \foreach \i in {1,2,3,4}{
            \draw (u\i) -- (ub\i) ;
        }
        \end{scope}

    \end{tikzpicture}
    \caption{\label{fig:recol-vertex-v}Recolouring sequence that unlocks vertices without changing the colour of $v$.}
    
\end{figure}

\medskip    

\noindent
\textbf{Case 2:} $z = u_2$. \\
By Claim~\ref{claim:unfreeze}, we can transform $\sigma$ to a colouring $\sigma'$ by only recolouring vertices in $V(f)$ so that $u_1$ is $\sigma'$-free and by recolouring $u_1$ if necessary we can further assume that $z$ is $\sigma'$-free. If $\sigma'(z) \not=a$ then we are done. Suppose $\sigma'(z) = a$. If $u_3$ is $\sigma'$-locked, then recolouring $z = u_2$ in turn unlocks $u_3$ by Lemma~\ref{rmk:unfreeze-propagate}, and we are done. If, on the other hand, $u_3$ is $\sigma'$-free, then we recolour $u_3$ if necessary to unlock $u_4$. If, at this point, $u_2$ is locked, then we can recolour $u_1$ which is still unlocked, and then $u_2$ becomes free. We finish the sequence by recolouring $z = u_2$ (to a colour distinct from $a$) and as $u_4$ is still free the claim follows. This completes Case 2. 

\medskip
 
 \noindent   
 \textbf{Case 3:} $z = u_3$, \\
As before, by Claim~\ref{claim:unfreeze}, we can transform $\sigma$ to a colouring $\sigma'$ by only recolouring vertices in $V(f)$ so that $u_3$ is $\sigma'$-free and $\sigma'(u_3) \not=a$.  Assume that each vertex in $\{u_2, \dots, u_5\}$ is $\sigma'$-locked (else we are done). 

Now recolouring $u_3$ unlocks both $u_2$ and $u_4$, and in turn recolouring $u_4$ unlocks $u_5$.  If, at this point, $u_3$ is locked, we recolour $u_2$ to unlock $u_3$ and then recolour $u_3$ if necessary so that its colour is distinct from $a$. Since $u_5$ is free, Case 3 is complete. 
%To complete the proof of the claim, by symmetry  if $z = u_4$ then we apply Case 3 and if $z = u_5$ then Case 2 applies.  
\end{proof}

We can now finish the proof of the Lemma~\ref{lem:5-face}. By the minimality of $(G, \alpha, \beta)$, there is a sequence $s'$ of recolourings from $\alpha_{\restriction (G - V(f))}$ to $\beta_{\restriction (G - V(f))}$.

We lift $s'$ to a sequence from $\alpha$ to $\beta$ in $G$ as follows. %Let $F$ be the set of vertices in $G - V(f)$ adjacent to some vertex in $V(f)$. 
Each time a vertex $v \in N(f)$ is recoloured to the current colour $a$ of some vertex $z \in V(f)$, we precede the recolouring of $v$ by changing the colour of $z$ to a colour distinct from $a$ and such that another vertex of $\{u_2,\ldots,u_5\}$ is unlocked; this is always possible by Claim \ref{claim:extend} since (by hypothesis), $f$ contains a free vertex. Observe that this operation only changes the colours of vertices in $V(f)$ and leaves $f$ with a free vertex in $\{u_2,\ldots,u_5\}$. (Note that $v$ indeed has at most one neighbour in $f$ by girth assumption). This shows that $s'$ lifts to a sequence from $\alpha$ to some colouring $\beta'$ of $G$ such that $\beta'_{\restriction (G - V(f))} = \beta_{\restriction (G - V(f))}$ and some vertex in $\{u_2, \dots, u_5\}$ is $\beta'$-free.  To finish the proof, we describe a sequence of recolourings from $\beta'$ to $\beta$ that recolours only vertices in $V(f)$. 

From $\beta'$ we recolour as many vertices as possible in $V(f)$ to colour $4$ (recall that $\beta$ uses only colours 1, 2 and 3) and let $\beta''$ denote the resulting colouring; notice that $\beta''(u_j) = \beta''(u_{j + 2}) = 4$ for some $j \in [5]$ (here $u_6 = u_1$ and $u_7 = u_2$). From $\beta$, we recolour $u_j$ and $u_{j + 2}$ to colour $4$ and denote the resulting colouring $\beta^*$. Note that $u_{j + 1}$ is free in both $\beta''$ and $\beta^*$ unless possibly if $u_{j + 1} = u_1$, in which case $\beta''(u_1) = \beta^*(u_1)$. Therefore, by recolouring $u_{j + 1}$ if necessary from $\beta''$, we can assume $\beta''(u_{j + 1}) = \beta^*(u_{j + 1})$. The only cause of difficulty is when $\beta^*(u_{j + 3}) = \beta''(u_{j + 4})$, $\beta^*(u_{j + 4}) = \beta''(u_{j + 3})$ and $u_{j + 4}$ and $u_{j + 3}$ are locked in $\beta^*$ (all the other cases are easy and left to the reader). In this case, assuming without loss of generality that $u_j$ is the vertex of degree $3$, we recolour, from $\beta^*$, $u_j$ to colour $\beta(u_j)$, then recolour $u_{j-1} = u_{j + 4}$ to colour $4$, followed by recolouring $u_{j -2} = u_{j + 3}$ to its colour $\beta^*(u_{j+4})$ in $\beta'$ and finally $u_{j+4}$ to $\beta^*(u_{j+3})$. We finish the sequence by recolouring $u_j$ to colour $4$. The proof of the lemma is complete. 
\end{proof}

%\subsection{5-vertices incident to 5-faces}

%We now consider \textcolour{magenta}{another} configuration and show that it is also cannot be present in $G$. Unlike the previous case, in this configuration we no longer have troubles due to locked colourings. However, the configuration that we consider is bigger, and consequently, vertices can have more than one neighbour in this configuration, which complicates some of the arguments.

In order to prove the next lemma, we make use of the following observation (see also Figure \ref{fig:frozen-5-face}). 

\begin{lemma}
\label{rmk:frozen-5face}
Let $f$ be a $5$-face in a plane graph $G$ with girth $5$ and $v\in V(f)$. Assume that all vertices of $f$ but $v$ have degree 3. If $d(v)=3$, there is no  $4$-colouring $\varphi$ of $G[V(f) \cup N(f)-N(v)]$ such that every vertex in $V(f)$ is $\varphi$-locked. Otherwise $d(v)\geqslant 4$, and there is a unique such colouring (see Figure~\ref{fig:frozen-5-face}).
\end{lemma}

%CF where did we use the lemma?
The (easy) proof of Lemma \ref{rmk:frozen-5face} follows by case analysis and is left to the reader. We are now ready to prove Lemma \ref{lem:config2-deg3}. 

\begin{figure}
    \centering
    \begin{tikzpicture}[thick,
    node/.style = {inner sep = 1pt},
    cnode/.style = {node, draw, circle}
    ]
        \node[cnode] (v) at (270:1) {4} ; 
        \node[cnode] (u1) at (198:1) {1} ;
        \node[cnode] (u2) at (126:1) {2} ;
        \node[cnode] (u3) at (54:1) {3} ;
        \node[cnode] (u4) at (342:1) {1} ;
        \draw (v) -- (u1) -- (u2) -- (u3) -- (u4) -- (v) ;
        \draw (252:2) -- (v) -- (288:2) ;
        \draw[fill=white] (270:1.9) ellipse (.8cm and 0.4cm);
        \node[node] (vb1) at (270:1.9) {2,3} ; 
        \node[node] (ub1) at (198:2) {3} ; 
        \node[node] (ub2) at (126:2) {4} ; 
        \node[node] (ub3) at (54:2) {4} ; 
        \node[node] (ub4) at (342:2) {2} ; 
        
        \foreach \i in {1,2,3,4}{
            \draw (u\i) -- (ub\i) ;
        }

    \end{tikzpicture}
    \caption{\label{fig:frozen-5-face}The unique locked $4$-colouring of a  $5$-face from Lemma~\ref{rmk:frozen-5face}. % with the given degree constraints, up to colour permutation and up to swapping the colours of the two bottom vertices.}
    }
\end{figure}

\begin{lemma}
\label{lem:config2-deg3}
Let $(G, \alpha, \beta)$ be a minimal counterexample. Then $G$ does not contain four $5$-faces $f_1, \dots, f_4$ such that $\bigcap_{i=1}^4 V(f_i)$ contains a vertex $v$ of degree $5$ and every vertex in $\bigcup_{i=1}^4 V(f_i) - \{v\}$ has degree $3$ (see Figure~\ref{fig:config2-deg3}). \end{lemma}

\begin{figure}[!ht]
\centering
\begin{tikzpicture}[thick,cnode/.style = {inner sep = 1pt, draw, circle},rotate=36]
\node[cnode,label=above:{$v$}] (v) at (0,0) {5};
\node[cnode,label=above:{$u_1$}] (v1) at (90:1) {3};
\node[cnode] (v2) at (162:1) {3};
\node[cnode] (v3) at (234:1) {3};
\node[cnode] (v4) at (306:1) {3};
\node[cnode,label=above:{$u_5$}] (v5) at (18:1) {3};
\node[cnode,label=above:{$z_1$}] (v12) at ($ (v1)+(126:1) $) {3};
\node[cnode,label=left:{$z'_1$}] (v12') at ($ (v2)+(126:1) $) {3};
\node[cnode] (v23) at ($ (v2)+(198:1) $) {3};
\node[cnode] (v23') at ($ (v3)+(198:1) $) {3};
\node[cnode] (v34) at ($ (v3)+(270:1) $) {3};
\node[cnode] (v34') at ($ (v4)+(270:1) $) {3};
\node[cnode,label=right:{$z_4$}] (v45) at ($ (v4)+(342:1) $) {3};
\node[cnode,label=above:{$z'_4$}] (v45') at ($ (v5)+(342:1) $) {3};
\draw (v2) -- (v) -- (v1) -- (v12) -- (v12') -- (v2) -- (v23) -- (v23') -- (v3) -- (v34) -- (v34') -- (v4) -- (v45) -- (v45') -- (v5) -- (v) -- (v4);
\draw (v) -- (v3);
\node at (126:1.25) {$f_1$};
\node at (198:1.25) {$f_2$};
\node at (270:1.25) {$f_3$};
\node at (342:1.25) {$f_4$};
\end{tikzpicture}
\caption{Configuration from Lemma~\ref{lem:config2-deg3}. Numbers indicate degree.}
\label{fig:config2-deg3}
\end{figure}

\begin{proof}
Suppose otherwise, and let $H = G[\bigcup_{i=1}^4 V(f_i)]$. For $i \in \{1, \dots, 4\}$ write $V(f_i) = \{v, u_i, u_{i+1}, z_i, z'_i\}$, where $u_i$ denotes the neighbour of $v$ incident with $f_{i-1}$ and $f_i$, and $u_iz_i, z_iz'_i,z'_i u_{i+1} \in E(f_i)$. %see also Figure~\ref{fig:todo}. % for a representation of this configuration.

By the minimality of $(G, \alpha, \beta)$, there is a sequence $s'$ of recolourings from $\alpha_{\restriction (G - H)}$ to $\beta_{\restriction (G - H)}$. To reach a contradiction, we show how to lift $s'$ to a sequence from $\alpha$ to $\beta$ in $G$. Fix a $4$-colouring $\sigma$ of $G$. 

%The proof will roughly go as follows: using the minimality of $G$, we can construct a recolouring sequence in $G \setminus H$, and then extend this recolouring sequence for the whole graph $G$, contradicting the assumption that $G$ is a counterexample. In order to extend this recolouring sequence to the whole graph $G$, a series of claims will be required:

\begin{claim}
\label{claim:config2-1}
%For every colouring $\sigma$ of $G$, and 
For $i = 1, 2$, there is a $\sigma$-free vertex incident to $f_i$, $f_{i+1}$ or $f_{i+2}$. 
\end{claim}
\begin{proof}[Proof of Claim \ref{claim:config2-1}.]

Suppose by contradiction that all vertices incident to $f_i, f_{i+1}$ and $f_{i+2}$ are $\sigma$-locked. By Lemma \ref{rmk:frozen-5face}, $\sigma(u_i) = \sigma(u_{i+1})$, $\sigma(u_{i+1}) = \sigma(u_{i+2})$ and $\sigma(u_{i+2}) = \sigma(u_{i+3})$. Thus only two colours appear in the neighbourhood of $v$, i.e.,  $v$ is $\sigma$-free, which is a contradiction. 
\end{proof}

\begin{claim}
\label{claim:config2-2}
There is a recolouring sequence from $\sigma$ to some colouring $\sigma'$ such that $\sigma_{\restriction (G - H)} = \sigma'_{\restriction (G - H)}$ and $\sigma(v) \neq \sigma'(v)$.
\end{claim}
\begin{proof}[Proof of Claim \ref{claim:config2-2}.]
By Claim~\ref{claim:config2-1}, there is a $\sigma$-free vertex $h$ incident  with $f_1$, $f_2$, or $f_3$. Let $P$ denote a shortest path whose internal vertices entirely lie in $V(f_1 \cup f_2 \cup f_3) - \{v\}$ and having ends $h$ and $u_1$. Applying Lemma~\ref{lem:extra} to $P$, we can assume $u_1$ is $\sigma$-free. Similarly, we can assume that $u_5$ is $\sigma$-free. (Note that this keeps $u_1$ free since $u_1$ has no neighbour in $V(f_2\cup f_3\cup f_4)-\{v\}$). This, in turn, yields $\sigma(u_1) = \sigma(u_5)$. Indeed, if they have different colours and none can be recoloured into the other, they can both be recoloured into $\{1,2,3,4\}\setminus\{\sigma(u_1),\sigma(u_5),\sigma(v)\}$. To prove the claim, we must change the colour of $v$. We consider two cases.

\medskip

\noindent
\textbf{Case 1:}  $f_2$ and $f_3$ are both $\sigma$-locked. \\
By Lemma \ref{rmk:frozen-5face},  $\sigma(u_2) = \sigma(u_3) = \sigma(u_4)$. Consequently, only two colours appear in the neighbourhood of $v$ since $\sigma(u_1)=\sigma(u_5)$ by assumption, i.e., $v$ is $\sigma$-free so we simply recolour $v$.      
\medskip

\noindent
\textbf{Case 2:} there is a $\sigma$-free vertex incident to either $f_2$ or $f_3$. \\ 
 By applying Lemma~\ref{lem:extra} as before, we can assume that $u_j$ is $\sigma$-free for some $j \in \{2, 3, 4\}$. Let $C$ be the set of colours appearing on the $\sigma$-locked neighbours of $v$. Then $|C| \leq 2$ by our observations this far and so there is a colour $a \in [4] \setminus (C \cup \{\sigma(v)\})$. We simply recolour the three $\sigma$-free neighbours $u_1$, $u_5$ and $u_j$ of $v$ if their colour is $a$ and finish the sequence by recolouring $v$ to $a$. This completes Case 2 and the proof of the claim. 
\end{proof}

\begin{claim}
\label{claim:config2-3}
There is a recolouring sequence from $\sigma$ to some colouring $\sigma'$ such that $\sigma_{\restriction (G - H)} = \sigma'_{\restriction (G - H)}$ and $\sigma'(v)\in\{\sigma'(z'_1),\sigma'(z_4)\}$. 
\end{claim}
\begin{proof}[Proof of Claim \ref{claim:config2-3}.]

Since $G$ is planar, there cannot be an edge between $z'_1$ and $\{z_3,z'_3\}$ and one between $z_4$ and $\{z_2,z'_2\}$. Up to exchanging $z'_1$ and $z_4$, assume that the former holds. 

Let $x$ be the neighbour of $z'_1$ which is not $z_1$ or $u_2$. By hypothesis $x\notin\{z_3,z'_3\}$, and since $G$ has girth 5, $x\notin\{z_2,z'_2\}$. If $x\notin H$, then by Claim~\ref{claim:config2-2}, we can assume that $\sigma(x) \neq \sigma(v)$. If $x=z_4$, then $\sigma(x)\neq\sigma(v)$ otherwise the result follows. If $x=z'_4$, we can also assume that $\sigma(x)\neq \sigma(v)$ otherwise $u_4$ is $\sigma$-free, and up to recolouring $u_4$, $x$ becomes free and we can recolour it. So we can assume that $\sigma(x) \ne \sigma(v)$.
%Since $G$ has girth $5$, $z'_1$ and $z_4$ cannot both have three neighbours in $V(H)$.  Assume without loss of generality that $z'_1$ has only two neighbours in $H$. By Claim~\ref{claim:config2-2}, we can assume that $\sigma(x) \neq \sigma(v)$, where $x$ is the neighbour of $z'_1$ not in $V(H)$.  

We try to immediately recolour  $z'_1$ to colour $\sigma(v)$. If this is not possible, then either $\sigma(z_1)$ or $\sigma(u_2)$ is $\sigma(v)$, and as $u_2 v \in E(G)$,  $\sigma(z_1) = \sigma(v)$. Therefore $u_1$ is $\sigma$-free, and up to applying Lemma~\ref{lem:extra}, we can assume that $z_1$ is $\sigma$-free; so we recolour $z_1$ to a colour different from $\sigma(v)$ and then finally recolour $z'_1$ with $\sigma(v)$. The claim is proved.
\end{proof}

\begin{claim}
\label{claim:config2-4}
For every $w \in V(G - H)$, there is a recolouring sequence from $\sigma$ to some colouring $\sigma'$ such that $\sigma_{\restriction (G - H)} = \sigma'_{\restriction (G - H)}$ and every vertex in  $N_G(w) \cap V(H)$  is $\sigma'$-free.  
\end{claim}
\begin{proof}[Proof of Claim \ref{claim:config2-4}.]
Set $Z_1 = \{u_1, z_1, z'_1\}$, $Z_2 = \{ z_2, z'_2\}$, $Z_3 = \{ z_3, z'_3\}$ and $Z_4 = \{ z_4, z'_4, u_5\}$. Note that $N_H(w) \subseteq \bigcup_{i=1}^4 Z_i$; moreover, since $G$ has girth $5$, $w$ can have at most one neighbour in each $Z_i$ that we denote $w_i$. 

By Claim~\ref{claim:config2-3}, we can assume that $\sigma(z'_1) = \sigma(v)$, in turn implying $u_2$ is $\sigma$-free. For $i \in [4]$, let $P_i$ denote the path in $H$ with ends $u_2$ and $w_i$ and that does not contain $v$. We apply Lemma \ref{lem:extra} to subpaths of each $P_i$ (with $P_i$ playing the role of $P$) starting with $P_4$ and working our way backwards to $P_1$ to obtain a sequence of recolourings that unlocks each $w_i$. At each step, $u_2$ stays $\sigma$-free since $z'_1$ and $v$ keep their colour. Therefore, if $w_i$ have not become free, one can always find a subpath of $P_i$ ending at $w_i$ and satisfying the hypotheses of Lemma~\ref{lem:extra}. This implies the claim. 
\end{proof}

We can now finish the proof of Lemma~\ref{lem:config2-deg3}.  By the minimality of $(G, \alpha, \beta)$, there is a sequence $s'$ of recolourings from $\alpha_{\restriction (G - H)}$ to $\beta_{\restriction (G - H)}$. 

We lift $s'$ to a sequence from $\alpha$ to some colouring $\beta'$ of $G$ by, for each vertex $w \in V(G - H)$, recolouring each neighbour $w'$ of $w$ in $H$ whenever $w$ gets recoloured to the colour of $w'$; this is possible by Claim \ref{claim:config2-4} and gives that $\beta'_{\restriction (G - H)} = \beta_{\restriction (G - H)}$. We now conclude the proof by recolouring $\beta'$ to $\beta$ in a similar fashion to Lemma \ref{lem:5-face}. Note that the only possible vertices coloured $4$ are in $H$ since $\beta$ is a $3$-colouring.
 
We consider two cases depending on whether $v$ has colour $4$ or not. Denote by $z'_0$ (resp. $z_5$) the neighbour of $u_1$ (resp. $u_5$) not in $\{v,z_1\}$ (resp. $\{v,z'_4\}$). Note that $z_0'$ and $z_5$ are not in $H$ since otherwise $G$ would contain a $C_4$.
\medskip

\noindent
\textbf{Case 1:} $v$ has colour $1,2$ or $3$. \\
We first show how to recolour $H$ so that all the $u_i$'s get colour $4$. Assume that this is not the case, and that there exists an index $i$ such that $u_i$ has colour 1, 2 or 3. We prove that we can recolour $H$ in such a way that either the number of $u_i$'s coloured with 4 increases, or it does not decrease but the number of $z_i$ and $z'_i$ coloured with 4 decreases. 

If $z_i$ and $z'_{i-1}$ are not coloured 4, then one can directly recolour $u_i$ with 4, which concludes. 

Otherwise, assume by symmetry that $z'_{i-1}$ has colour $4$. Then let $j<i$ be the largest index such that $z'_{j-1}$ is not coloured $4$ (which exists since $z'_0$, which is not in $H$, has colour $1, 2$ or $3$). Now $z'_j$ has colour $4$, hence $z_j$ does not. Now $u_j$ can be recoloured to $4$, which increases the number of $u_\ell$'s coloured with $4$. So we can assume that $u_j$ is coloured $4$. Hence $z_j$ is free (since it has two neighbours coloured with $4$) and then can be recoloured with a colour distinct from $4$. So either we can immediately recolour $z'_j$ or recolour it after the recolouring of $z_j$. So we can recolour $H$ in such a way $z_j'$ is not coloured $4$, which does not change the number of vertices $u_\ell$ coloured with $4$, but the number of $z_\ell,z'_\ell$ coloured with $4$ decreased. %\textcolour{red}{N. Small bug there if $z_j=z_5$ ?}

By iterating this argument, we obtain a colouring of $H$ where colour $4$ lies exactly on the $u_i$'s. We now claim that the following $(\star)$ property holds: if $z_i,z_i'$ are both locked, we can flip the colours of $z_i$ and $z_i'$ by only recolouring vertices among $\{v,u_i,u_{i+1},z_i,z_i'\}$. Indeed, if that is the case, the neighbours of $z_i$ and $z'_i$ outside of $\{v,u_i,u_{i+1},z_i,z'_i\}$ have the same colour. So we can recolour $v$ with the colour of $z_i$, then recolour $u_i$ with a free colour, then $z_i$ with $4$, then $z_i'$ with the previous colour of $z_i$. Now, up to recolouring $v$ (which only has two colours in its neighbourhood) and freeing $u_i$, we can colour $z_i$ with the initial colour of $z_i'$ and put back colour 4 on $u_i$.

Now let us prove that for increasing $i \le 4$, we can colour $z_i,z_i'$ with their target colour. Note that these target colours do not appear on $u_i,u_{i+1}$ since they are coloured with 4 and $\beta$ is a 3-colouring. If they appear on another neighbour in $H$, it should be in $\{z_j,z'_j\}$ for some $j \ne i$. Moreover there is at most one edge between $\{z_i,z'_i\}$ and $\{z_j,z'_j\}$ since $G$ has girth $5$. Now, using property $(\star)$ if needed, we can recolour $z_j,z'_j$ so that they are not coloured anymore with the target colour of their neighbour in $\{z_i,z'_i\}$. Note that this happens only if $j>i$ since $\beta$ is a proper colouring and the vertices $z_j,z'_j$ with $j<i$ already have their final colour. In particular, these recolouring steps do not recolour the vertices $z_j,z'_j$ for $j<i$.

So we can assume that the final colour of $z_i$ (resp. $z'_i$) does not appear on its neighbourhood except maybe on $z'_i$ (resp. $z_i$). So we can directly colour $z_i,z'_i$ with their target colour unless both are locked and their target colours are then flipped. But then we apply again $(\star)$ to swap their colours.

We finally set $v$ to $\beta(v)$, then each $u_i$ to $\beta(u_i)$.

\medskip

\noindent
\textbf{Case 2:} $v$ has colour $4$. \\
We first recolour vertices to 4 as long as it is possible. We then claim that if $u_i$ and $u_{i+1}$ are locked, we can recolour them (together with $z_i,z'_i$) so that they have the same colour. Indeed, assume that $u_i$ is coloured with 1 and $u_{i+1}$ with 2. Note that both $z_i$ and $z'_i$ must have a neighbour with colour 4 by maximality of $\beta'$. Now either $z_i$ is coloured with 2 or $z'_i$ is coloured with 1 (since both cannot have colour 3). By symmetry we consider only the latter. Either $z_i$ is coloured 2 and we recolour it with 3 so that $u_i$ can be recoloured with 2, or $z_i$ is coloured 3, and we recolour it to 2, so that $z'_i$ can be recoloured to 3 and $u_{i+1}$ to 1. Afterwards, we recolour again vertices with 4 whenever possible. 

Therefore, if colours 1, 2 and 3 appear on locked $u_i$'s, it must be on $u_1,u_3$ and $u_5$. By maximality of $\beta'$, the vertices coloured with 4 among $\{z'_1,z_2,z'_3,z_4\}$ must dominate $\{z_1,z'_2,z_3,z'_4\}$. Due to planarity and girth constraints, this implies that at least three vertices among $\{z'_1,z_2,z'_3,z_4\}$ are coloured with 4. Therefore, either $u_2$ or $u_4$ (say $u_2$ by symmetry) sees only colour 4 in its neighbourhood. In that case, either $z_1$ sees two vertices coloured with 4, and we can recolour it, which unlocks $u_1$, or we recolour $u_2$ so that $z'_1$ becomes free, then recolour $z'_1$ and finally recolour $z_1$ with 4. In both cases, only $u_3$ and $u_5$ can be locked in the obtained colouring.

Therefore, we can assume that at most two colours appear on the locked $u_i$'s. Let $c$ be one of the remaining colours. Now, observe that all the $u_i$'s coloured with $c$ are unlocked, hence they can be recoloured. We may now recolour $v$ to $c$, so that we can apply Case~1.
 %Let us consider the $j^{th}$ step of the recolouring sequence $\mathcal S$, and assume that we managed to extend this sequence up to this step. Let $\sigma$ be the corresponding colouring that we obtained, and $w$ be the vertex recoloured at this step in $\mathcal S$, and $c$ be the colour $w$ is recoloured with. If $w$ has no neighbours in $H$, then we can apply this recolouring step directly. Otherwise, by Claim~\ref{claim:config2-4}, there is a colouring $\sigma'$ and a recolouring sequence from $\sigma$ to $\sigma'$ such that the two colourings agree on $G \setminus H$, and the neighbours of $w$ in $H$ are all locally non-frozen. Then, from $\sigma'$ we can recolour these neighbours with a colour different from $c$, and then recolour $w$ with the colour $c$.
%Hence we obtain a recolouring sequence in $G$ from $\alpha$ to a colouring $\beta'$ which agrees with $\beta$ on $G \setminus H$. We can then easily finish the transformation from $\beta'$ to $\beta$ since $\beta$ is a 3-colouring of $G$.
\end{proof}

%\subsection{Frozen vertices}

Let $G$ be a plane graph with girth $5$, let $v$ be a vertex of $G$ of degree $4$ and let $f$ be a $5$-face incident with $v$. A $5$-face $f' \not=f$ of $G$ is \emph{opposite to $f$ with respect to $v$} if $V(f) \cap V(f') = \{v\}$ (note that since $v$ has degree $4$ and $G$ has girth $5$, $f'$ is well defined). We say that a vertex $u$ incident with a $5$-face $f$ is \emph{bad for $f$} if either \begin{itemize}
\item $u$ has degree $3$,  or 
\item $u$ has degree $4$ and $f$ has an opposite $5$-face $f'$ with respect to $u$ such that each vertex in $V(f') - \{u\}$ is bad for $f'$. 
\end{itemize}

If $v$ is bad for $f$, we say that $(v,f)$ is a \emph{bad pair}. Observe that bad pairs are well-defined inductively, and that the definition gives a natural quasi-order on bad pairs. We slightly abuse terminology by saying that a vertex is \emph{bad} if it is bad for at least one of its incident faces. Bad vertices have the following property.

\begin{lemma}
\label{lem:bad-frozen}
Let $(G, \alpha, \beta)$ be a minimal counterexample. If $v$ is a bad vertex of $G$ for some $5$-face $f$ and has degree $4$, then $v$ is $\alpha$-frozen and both neighbours of $v$ which are not incident to $f$ are also $\alpha$-frozen and are coloured alike under $\alpha$.
\end{lemma}
\begin{proof}
 We prove the result by induction on the pairs $(v,f)$ such that $v$ has degree 4 and is bad for $f$. Let $v$ be a vertex of degree $4$ which is bad for the face $f$. %First assume that $v$ is a smallest degree 4 vertex marked as bad. Let $f'$ be the face opposite of $f$ at $v$. By definition, all the vertices different from $v$ incident to $f'$ are bad. Since $v$ is the first degree 4 vertex marked as bad, this implies that $f'$ is incident to $4$ vertices of degree 3, and one vertex of degree 4, i.e. $f'$ is very bad. By Lemma~\ref{lem:5-face}, this implies that all the vertices incident to $f'$ are frozen, including $v$ and its two neighbours incident to $f'$. Moreover, by Remark~\ref{rmk:frozen-5face}, the two neighbours of $v$ incident to $f'$ have the same colour.

Let us assume that all the bad pairs smaller than $(v,f)$ satisfy the statement of the lemma. Let $f'$ be the face opposite of $f$ at $v$. 
\begin{claim}
All the vertices of $f'$ are frozen, and if $z\neq v$ is incident to $f'$, then its two neighbours in $f'$ have different colours.
\end{claim}
\begin{proof}
In order to prove the first part of this claim, we show that:
\begin{itemize}
    \item $f'$ contains at least one frozen vertex; 
    \item if $z$ is a frozen vertex of $f'$ different from $v$, then its neighbours are also frozen.
\end{itemize}
If all the vertices of $f'$ different from $v$ have degree 3, then $f'$ is very bad and the result of the claim follows immediately from Lemma~\ref{lem:5-face}. On the other hand, vertices of degree $4$ different from $v$ and incident to $f'$ are frozen using the induction hypothesis, hence the first point holds.

In order to prove the second point, consider a frozen vertex $z$ incident to $f'$ and different from $v$. If $z$ has degree $3$, then the neighbours of $z$ are frozen by Lemma~\ref{lem:frozen3} and in particular its neighbours on $f'$ have different colours. If $z$ has degree $4$, then since we know that $z$ is bad for $f'$, by the induction hypothesis, the two neighbours of $z$ which are not incident to $f'$ are frozen and have the same colour. By Lemma~\ref{lem:frozen4}, this implies that all the neighbours of $z$ are frozen. Moreover, since the neighbours of $z$ outside of $f'$ have the same colour, its neighbours incident to $f'$ have different colours.
\begin{comment}
\textcolour{blue}{Let us now consider the second part of the claim. If $z$ is a vertex incident to $f'$ different from $v$, then:
\begin{itemize}
    \item either $z$ has degree $3$, and since $z$ is locked, then all its neighbours have different colours; 
    \item or $z$ has degree $4$, and $z$ is bad for $f'$. Using the induction hypothesis, we have that $z$ is locked and the two neighbours of $z$ not incident to $f'$ have the same colours. Hence, the two neighbours of $z$ incident to $f'$ must have different colours since otherwise $z$ would not be locked.
\end{itemize}
} \textcolour{magenta}{This appears to be a lost artifact}
\end{comment}
\end{proof}

Hence, the claim above shows that $v$ is frozen. Moreover, the two neighbours of $v$ incident to $f'$ must have the same colour since otherwise it is not possible to colour the vertices of $f'$ with $4$ colours such that the property of the claim above holds. This proves the induction step and finishes the proof of the lemma.
\end{proof}

We now have all the tools needed to forbid the structures presented in Subsection~\ref{sub:defs}.

\begin{lemma}
\label{lem:config1}
The graph $G$ does not contain a 5-face $f$ such that every vertex on $f$ is bad for $f$.
\end{lemma}
\begin{proof}
Let us assume by contradiction that $G$ has a face $f$ satisfying this property.
By Lemmas~\ref{lem:5-face} and~\ref{rmk:frozen-5face}, $f$ contains a vertex of degree 4, and all of them are frozen by Lemma~\ref{lem:bad-frozen}. Moreover, as before, if $z$ is a frozen vertex of $f$, then its neighbours are also frozen. This follows from Lemma~\ref{lem:frozen3} if $z$ has degree $3$, and from Lemma~\ref{lem:frozen4} and Lemma~\ref{lem:bad-frozen} if $z$ has degree $4$. Let $u$ and $v$ be two vertices incident to $f$. We will show that $u$ and $v$ have different colours. This is trivially true if $u$ and $v$ are adjacent. If they are not adjacent, then they have a common neighbour $w$ incident to $f$. If $w$ has degree $3$, then since $w$ is locked all its neighbours have different colours, and in particular $\alpha(u) \neq \alpha(v)$. If $w$ has degree $4$, then by Lemma~\ref{lem:bad-frozen}, its two neighbours not incident to $f$ have the same colour, and consequently $u$ and $v$ must have different colours. Hence, all the vertices incident to $f$ have different colours in $\alpha$, which contradicts the assumption that $\alpha$ is a $4$-colouring of $G$.
\end{proof}

An important consequence of Lemma~\ref{lem:config1} is stated in the following corollary.
\begin{corollary}
\label{cor:2bad}
Let $v$ be a vertex of $G$ of degree $4$, incident to the faces $f_1,f_2,f_3,f_4$ (in clockwise order). Then $v$ cannot be bad for both $f_i$ and $f_{i+2\mod 4}$. In particular, every bad 4-vertex is bad for at most two faces.
\end{corollary}
\begin{proof}
Let us assume by contradiction that $v$ is bad for both $f_1$ and $f_3$. Since $v$ is bad for $f_1$, then by definition all the vertices of $f_3$ different from $v$ are bad for $f_3$. Since $v$ is also bad for $f_3$, we obtain a contradiction with Lemma~\ref{lem:config1}.
\end{proof}

\begin{lemma}
\label{lem:config2}
The graph $G$ does not contain a 5-vertex $v$ adjacent to four 5-faces $f_1,f_2,f_3,f_4$ such that, for each $i$, all vertices incident to $f_i$ except $v$ are bad for $f_i$.
\end{lemma}
\begin{proof}
Let us assume by contradiction that $G$ contains a vertex $v$ with this property. Let $H$ be the graph induced by the vertices incident to the faces $f_i$. The case where all the vertices of $H$ different from $v$ have degree $3$ has already been handled in Lemma~\ref{lem:config2-deg3}. Hence we can assume that at least one vertex incident to $v_i$ has degree $4$. By Lemma~\ref{lem:bad-frozen} we know that this vertex is frozen, and so are its neighbours. Moreover, by Remark~\ref{lem:frozen3}, all the neighbours of a frozen vertex of degree $3$ are also frozen. Hence, it follows that all the vertices of $H$ are frozen. However, this implies that the colouring $\alpha$ is locked on $H$, which is not possible by Claim~\ref{claim:config2-1}.
\end{proof}

\subsection{Discharging}

\begin{comment}
\textcolour{red}{Probably move this definition somewhere else, and reformulate.}
Let $f$ be a face, and $v$ a vertex of degree $4$ incident to $f$. We say that $f'$ is the \emph{opposite face} of $f$ with respect to $v$, and $f_1,f_2$ are the \emph{side faces} of $f$ with respect to $v$. \textcolour{blue}{Note that since a vertex can be bad for several faces, the same face can be of several types, even for a fixed vertex on it.} (\textcolour{magenta}{I think we can just get rid of the sentences I put in blue.})

\textcolour{magenta}{Perhaps move this to the start of the frozen section and phrase it as: Let $f$ be a face, and $v$ a vertex of degree $4$ incident to $f$. Let $f_{1},f_{2},f_{3},f_{4}$ be the four faces incident with $v$ in clockwise order where $f_{1} = f$. We say that $f_{3}$ is the \emph{opposite face} of $f$ with respect to $v$, and $f_{2},f_{4}$ are the \emph{side faces} of $f$ with respect to $v$. }
\end{comment}

We may now reach a contradiction using a discharging argument. We first give an initial weight of $2d(v)-6$ to each vertex $v$ and $\ell(f)-6$ to each face $f$. According to Euler's formula, the total weight is  
\[\sum_{v\in V(G)} (2d(v)-6)+\sum_{f\in F(G)}(\ell(f)-6)=-12.\]
We then redistribute the weights according to the following rules :
\begin{itemize}
\item Every 4-vertex which is bad for at least two faces gives 1 to their opposite faces.
\item Every 4-vertex which is bad for only one face $f$ gives 1 to its opposite face, and then splits its remaining weight equally among its two remaining incident 5-faces.
\item Every 4-vertex which is not bad and every $5^+$-vertex $v$ gives 1 to each incident 5-face $f$ such that all the vertices of $f\setminus\{v\}$ are bad, and $\frac12$ to every other 5-face. 
\end{itemize}

We finally reach a contradiction by showing that every vertex and face ends up with non-negative weight after this process. This contradicts that the sum of the weights is negative and proves Theorem~\ref{thm:4col-girth5}. First recall that every vertex of $G$ has degree at least 3, and that 3-vertices and $6^+$-faces have non-negative initial weight and do not lose any weight. Therefore, we only have to consider $4^+$-vertices and $5$-faces.

We may easily take care of vertices of degree $d\geqslant 6$: they give at most 1 to each incident face, hence they end up with at least
\[2d-6-d=d-6\geqslant 0.\]

\begin{lemma}
Every 5-vertex ends up with non-negative weight.
\end{lemma}

\begin{proof}
By Lemma~\ref{lem:config2}, we know that every 5-vertex gives 1 to at most 3 faces and at most $\frac12$ to every other face, hence their final weight is at least $4-3\times 1-2\times\frac12=0$.
\end{proof}

\begin{lemma}
Every 4-vertex ends up with non-negative weight.
\end{lemma}

\begin{proof}
Let $v$ be a 4-vertex and $f_1,f_2,f_3,f_4$ be the faces incident to $v$. If $v$ is bad, then by Corollary~\ref{cor:2bad}, it is bad for at most two faces. Then either $v$ is bad for two faces and loses $2$ by the first rule, or it is bad for one face, and distributes all its weight by the second rule. In both cases, bad 4-vertices end up with weight 0.

We may thus assume that $v$ is not bad. If $v$ gives 1 to an incident face, say $f_1$, then $f_3$ cannot have length 5 otherwise $v$ would be bad for $f_3$, and $v$ would not give any weight to $f_3$. Otherwise, $f$ gives at most $2\times \frac12$ to $f_1$ and $f_3$. Therefore, $v$ gives at most 1 to $\{f_1,f_3\}$, and at most 1 to $\{f_2,f_4\}$ by symmetry. Then the final weight of $v$ is at least $2-2\times 1=0$. This completes the proof of the lemma. 
\end{proof}

\begin{lemma}
Every 5-face receives a weight of at least 1.
\end{lemma}

\begin{proof}
Let $f$ be a 5-face. By Lemma~\ref{lem:config1}, $f$ contains at most four bad vertices for $f$. Let $v$ be a vertex on $f$ which is not bad for $f$. If $v$ has degree at least $5$ or is a 4-vertex which is not bad, then it gives at least $\frac12$ to $f$ by the third rule. If $v$ has degree 4 and is bad, let $f,f_1,f',f_2$ be the faces incident to $v$ in clockwise order. Observe that if $v$ is bad for $f'$, then it gives 1 to $f$ since $f$ is opposite to $f'$ with respect to $v$. Otherwise, since $v$ is bad but not bad for $f$ by construction, then $v$ is bad for $f_1$ or $f_2$ (but not both since $f_1$ and $f_2$ are opposite with respect to $v$, see Corollary~\ref{cor:2bad}). Then the second rule applies, and $f$ % is a side face of $f_1$ (or $f_2$) with respect to $v$, hence 
receives at least $\frac12$ from $v$. 

Therefore, every vertex on $f$ which is not bad for $f$ gives at least $\frac12$ to $f$. We may thus assume that $f$ has only one such vertex $v$. In that case, $f$ receives 1 from $v$ by the third rule unless $v$ is a bad 4-vertex. In that case, again let $f,f_1,f',f_2$ be the faces incident to $v$ in clockwise order. If $f'$ has length 5, then since $v$ is the only vertex in $f$ which is not bad for $f$, it implies that $v$ is bad for $f'$ and thus gives 1 to $f$ by the first or second rule since $f$ is a face opposite to $f'$ with respect to $v$. Otherwise, $f'$ has length at least 6, and $v$ is bad for $f_1$ or $f_2$ (but not both, again by Corollary~\ref{cor:2bad}). In that case, $f$ %is the only \textcolour{magenta}{side} 5-face of $v$, hence it 
receives also 1 by the second rule. This completes the proof of the lemma.
\end{proof}

\bibliographystyle{abbrv}
\bibliography{biblio.bib}

\end{document}